\documentclass[10pt,a4paper,twoside]{amsart}

\usepackage{amssymb}
\usepackage{amsthm}
\usepackage{amsmath}
\usepackage{a4wide}
\usepackage[english]{babel}
\usepackage{hyperref}
\usepackage{pstricks,graphicx,bbm}

\begin{document}

\theoremstyle{plain}
\newtheorem{theorem}{Theorem}[section]
\newtheorem{proposition}{Proposition}[section]
\newtheorem{lemma}{Lemma}[section]
\newtheorem{corollary}{Corollary}[section]
\theoremstyle{definition}
\newtheorem{definition}{Definition}[section]
\theoremstyle{remark}
\newtheorem{remark}{Remark}[section]
\numberwithin{equation}{section}
\setlength{\parindent}{0pt}

\def\red{\color{red}}

\title{The dynamics of critical fluctuations in asymmetric Curie-Weiss models}

\author{Paolo Dai Pra \and Daniele Tovazzi}


\keywords{Interacting Particle Systems, Mean-field interaction, Averaging Principle}

\subjclass[2010]{60K35, 82C44}

\date{\today}

\maketitle

\begin{abstract}
We study the dynamics of fluctuations at the critical point for two time-asymmetric version of the Curie-Weiss model for spin systems that, in the macroscopic limit, undergo a Hopf bifurcation. The fluctuations around the macroscopic limit reflect the type of bifurcation, as they exhibit observables whose fluctuations evolve at different time scales. The limiting dynamics of fluctuations of slow observable is obtained via an averaging principle.

\end{abstract}

\section{Introduction}

Systems of many interacting particles exhibit peculiar behaviors as they get close to a phase transition. The phenomena occurring on this regime, referred to as {\em critical phenomena}, include long range correlations and large, non normal fluctuations. The interest in critical phenomena has been strongly stimulated by the celebrated article \cite{BTW}, where it is shown that certain interacting systems are {\em spontaneously} attracted by  their critical point ({\em self-organized criticality}). 

Either emerging from self-organization or from tuning model's parameters, critical phenomena are usually hard to treat at rigorous mathematical level; in part for this reason, considerable attention has been directed to mean-field models, whose tractability may allow to detect some universal features in criticality. In this paper we continue the analysis of critical fluctuations in mean-field dynamics. As first shown in \cite{D83, CE88}, reversible mean-field dynamics with ferromagnetic interaction have fluctuations at the critical point that are non-normal, and with an anomalous space-time scaling. The common features of the models considered is that the macroscopic dynamics, given by the {\em McKean-Vlasov equation}, exhibits a {\em pitchfork bifurcation} at the critical point: in particular, a single stable equilibrium bifurcates into two distinct locally stable equilibria, corresponding to the magnetized phases. The nature of the bifurcation actually matters, as the dynamics of fluctuations is related to the linearization of the McKean-Vasov equation. Results in the same spirit have been recently obtained in \cite{CerfGorny} for a class of models in which criticality is achieved by self-organization (see \cite{CerfGorny2} for related result in equilibium).

In \cite{CDP12} the effects of quenched disorder on critical fluctuations have been investigated in two specific examples: the Curie-Weiss model and the mean-field Kuramoto model. For the Curie-Weiss model, the disorder takes the form of a random, site-dependent magnetic field; it brakes space homogeneity of the system, but it maintains its time symmetry. The nature of the bifurcation is the same as in the homogeneous model; however the scale of critical fluctuations, as well as their distribution, drastically changes, as disorder's fluctuations become dominant. In the Kuramoto model the disorder is the random characteristic frequency of each rotator; this induces a preferential direction of rotation at microscopic level, which is clearly not invariant by time reversal. At macroscopic level this may change the nature of the bifurcation in the McKean-Vlasov equation. When the intensity of the disorder is small the bifurcation is still of pitchfork type, and the disorder turns out to have only moderate effects on the critical fluctuations. For larger intensity of the disorder the bifurcation changes nature, becoming of {\em Hopf} type: the emergence of stable periodic orbit is expected, although not fully rigorously proved (\cite{Bon1, Bon2}). Critical fluctuations have not been yet described in this regime.

In this paper we study critical fluctuations for two models, obtained by modifying the Curie-Weiss model, in which the bifurcation at the critical point is of Hopf type. In the first model the classical Curie-Weiss dynamics is modified by introducing dissipation, as proposed in \cite{DPFR13}; in the second we consider a two-population version of the Curie-Weiss model that has been studied in \cite{GC,CGM, CFT16}. In both examples the analysis leads to the study of the evolution of a two-dimensional order parameter. After a change of variables, we identify a slow and a fast variable: in the ``natural'' time scale, the fast variable averages out, producing a limiting dynamics for the slow variable via an averaging principle.

In Section 2 we formally introduce the two models and state our main results. Proofs are then given in Sections 3 and 4.

\section{Models and main results}
\subsection{The Curie-Weiss model with dissipation}\label{CW_dissipated_model_and_result}
\subsubsection{Description of the model}
Let $\mathcal{S}=\{-1,+1\}$ and $\underline{\sigma}=\left(\sigma_i\right)_{i=1}^N\in\mathcal{S}^N$ be a configuration of $N$ spins. We can define, at least at informal level, a stochastic process $\left(\underline{\sigma}(t)\right)_{t\in[0,T]}$ by assigning (besides an initial condition) the spin flip rates. Let us denote with $\underline{\sigma}^i$ the configuration obtained by $\underline{\sigma}$ by flipping the $i$-th spin, namely $$\sigma^i_k=\left\{ \begin{array}{cc}
\sigma_i, & i\not=k,\\-\sigma_i, & i=k.
\end{array}\right.$$
At a given time $t\in[0,T]$, if $\underline{\sigma}(t)=\underline{\sigma}$, each transition $\sigma_j\to-\sigma_j$ occurs with rate $1-\tanh(\sigma_j \lambda_N)$, where $\lambda_N$ is a stochastic process evolving according to the stochastic differential equation \begin{equation}
\label{lambda} d\lambda_N(t)=-\alpha\lambda_N(t)dt+\beta dm_N(t),
\end{equation} where $\alpha,\beta> 0$ and \begin{equation}
\label{magne} m_N(t)=\frac{1}{N}\sum_{j=1}^N \sigma_j(t). 
\end{equation}
Formally speaking, we are dealing with a Markov process $(\underline{\sigma}(t),\lambda_N(t))\in\mathcal{S}^N\times \mathbb{R}$ whose infinitesimal generator is \begin{equation}
\label{infgenCWdiss} L_Nf(\underline{\sigma},\lambda)=\sum_{i=1}^{N}\left[ \Big(1-\tanh(\sigma_i \lambda)\Big)\left(f\left(  \underline{\sigma}^i, \lambda-\frac{2\beta\sigma_i}{N} \right) -f(\underline{\sigma},\lambda)\right) \right] -\alpha\lambda f_\lambda(\underline{\sigma},\lambda).
\end{equation}
This is a simplified version of the model introduced in \cite{DPFR13}. The expression \eqref{infgenCWdiss} describes a system of mean field ferromagnetically coupled spins, in which the interaction energy is dissipated over time. The parameter $\beta$ represents the inverse temperature while $\alpha$ describes the intensity of dissipation in the interaction energy: notice that by setting $\alpha=0$ we would obtain a Glauber dynamics for the classical Curie-Weiss model. In the following, as initial condition we will take the spins $\{\sigma_i(0)\}_{i\in\mathbb{N}}$ as a family of i.i.d.  symmetric Bernoulli random variables and $\lambda_N(0)=\lambda_0\in \mathbb{R}$.
\subsubsection{Limiting Dynamics}
We want to study the dynamics of the process defined by $\eqref{infgenCWdiss}$ in the limit as $N\to\infty$ in a fixed time interval $[0,T]$. 
Consider the \textit{empirical measure} flow, for $t\in[0,T]$: $$\rho_N(t)=\frac{1}{N}\sum_{j=1}^N \delta_{\{\sigma_j(t),\lambda(t)\}}.$$
Given a measurable function $f:\mathcal{S}\times\mathbb{R}\to\mathbb{R}$, we are interested in the asymptotic behaviour of \emph{empirical averages} of the form $$\int_{\mathcal{S}\times\mathbb{R}}fd\rho_N(t)=\frac{1}{N}\sum_{j=1}^N f(\sigma_j(t), \lambda(t)). $$
By \textit{order parameter} we mean a stochastic process, defined as an empirical average, whose dynamics are Markovian: the dynamics of an order parameter completely describes the dynamics of the original system. In our case, we can find a two-dimensional order parameter (see Lemma \ref{orderparameter}) and its limiting dynamics is described by the following theorem, .

\begin{theorem}\label{limitingdyntheorem}For $t\in [0,T]$, the process $(m_N(t),\lambda_N(t))\in [-1,1]\times \mathbb{R}$, defined by \eqref{lambda} and \eqref{magne}, is an order parameter of the system.
As $N\to\infty$, $(m_N(t),\lambda_N(t))_{t\in[0,T]}$ converges, in sense of weak convergence of stochastic process, to a limiting deterministic process, solution of the system of ordinary differential equations \begin{equation}
\label{limitingdyn}  \begin{cases}  \dot{m}(t)\!\!\!\!\!\!\!\!\!\!\!\!&=2\left(\tanh(\lambda(t))-m(t)\right), \\ \dot{\lambda}(t)&=2\beta\left(   \tanh(\lambda(t))-m(t)\right)-\alpha\lambda(t), \end{cases}  
\end{equation} with initial conditions $m(0)=0, \: \lambda(0)=\lambda_0$.
\end{theorem}
\noindent
We briefly recall the analysis of the limiting system performed in Section 3 of \cite{DPFR13}:
 \eqref{limitingdyn} admits a unique stationary solution $(0,0)$ for any choice of the parameters $\alpha,\beta$. Anyway, for $\beta\leq {\alpha\over 2}+1$ the origin is a global attractor, while, for $\beta>{\alpha\over 2}+1$, $(0,0)$ loses its stability and system \eqref{limitingdyn} has a unique periodic orbit, which attracts all trajectories except the fixed point. In the critical case $\beta={\alpha\over 2}+1$, an Hopf bifurcation occurs.
\subsubsection{Normal Fluctuations}
Theorem \ref{limitingdyntheorem} configures as a Law of Large Number for the empirical measure flow $(\rho_N(t))_{t\in[0,T]}$, then it is natural to wonder whether a Central Limit Theorem holds as well. Let $(q_t)_{t\in[0,T]}$ denote the limiting dynamics of $(\rho_N(t))_{t\in[0,T]}$: in this paragraph we study the fluctuation flow $\tilde{\rho}_N(t)$, where
\begin{equation}
\label{fluctuationflowCWdiss} \tilde{\rho}_N(t)=N^{1\over 2} (\rho_N(t)-q_t)
\end{equation}
for any $t\in[0,T]$. Obviously, $(q_t)_{t\in[0,T]}$ is described by the solution of \eqref{limitingdyn} $(m(t),\lambda(t))_{t\in[0,T]}$, hence an order parameter for \eqref{fluctuationflowCWdiss} is given by the process $(\tilde{m}_N(t),\tilde{\lambda}_N(t))_{t\in[0,T]}$ where 
$$\tilde{m}_N(t)=N^{1\over 2}(m_N(t)-m(t)), \:\:\:\:\:\:\:\: \tilde{\lambda}_N(t)=N^{1\over 2}(\lambda_N(t)-\lambda(t)).$$
As one may expect, the order parameter converges to a Gaussian bi-dimensional process, as stated in the following theorem, 
that follows from standard {\em propagation of chaos} arguments.
\begin{theorem}\label{NormalFlucCWdiss}
For any $\alpha>0$, $\beta>0$ and $t\in [0,T]$ the process $(\tilde{m}_N(t),\tilde{\lambda}_N(t))$ converges, in sense of weak convergence of stochastic processes, as $N\to \infty$ to the Gaussian process $(\tilde{m}(t),\tilde{\lambda}(t))$, unique solution of the linear time-inhomogenous stochastic differential equation \begin{equation}\label{CLT}d\left(\begin{array}{c}
\tilde{m}(t)\\ \tilde{\lambda}(t)
\end{array}\right)=A(t)\left(\begin{array}{c}
\tilde{m}(t)\\ \tilde{\lambda}(t)
\end{array}\right)dt+\sqrt{1-m(t)\tanh(\lambda(t))}\left(\begin{array}{c}
2\\ 2\beta
\end{array}\right) dB(t)\end{equation}
with $$
A(t)=\left(\begin{array}{cc}
-2 & 2(1+\tanh(\lambda(t)) \\ -2\beta(1-\tanh(\lambda(t)) & 2\beta-\alpha
\end{array}\right),
$$
$B(t)$ one-dimensional standard Brownian Motion and $\tilde{m}(0)=0$, $\tilde{\lambda}(0)=0$.
\end{theorem}

\subsubsection{Dynamics of critical fluctuations}
The result of Theorem \ref{NormalFlucCWdiss} holds for any regime, but our main goal is to study more closely the long-time behaviour of fluctuations at the critical point, since typically they display some peculiar features (see \cite{CDP12}, \cite{CE88} and \cite{D83}). From now on, we will always take the parameters $\alpha$ and $\beta$ in such a way $\beta=\frac{\alpha}{2}+1$.\\
Let us consider the critical fluctuation flow for $t\in[0,T]:$
$$
\hat{\rho}(t)=N^{1\over 4}(\rho_N(N^{1\over 2})-q^*_0),
$$ 
where $q^*_0$ denotes the stationary solution correspondent to $(0,0)$, the equilibrium point of \eqref{limitingdyn}. In this way, we are assuming that the process starts in local equilibrium, which simplifies the proof of our result, but it should not be difficult to extend it to a general initial condition. Notice also that we are employing the usual space-time scaling involved in critical fluctuations. The flow $(\hat{\rho}_N(t))_{t\in[0,T]}$ can be fully described by the order parameter:
 \begin{equation}\label{wz} \hat{m}_N(t)=N^{1\over 4}m_N(N^{1\over 2}t), \:\:\:\:\:\: \hat{\lambda}_N(t)=N^{1\over 4}\lambda_N(N^{1\over 2}t).\end{equation} After having performed the change of variable (see Subsection \ref{pre}) \begin{equation}
\label{u} 
\begin{cases} z_N(t)=\hat{\lambda}_N(t),\\ u_N(t)=\frac{\beta \hat{m}_N(t)-\hat{\lambda}_N(t)}{\sqrt{\beta-1}},\end{cases}
\end{equation} consider the process $\kappa_N(t)=z_N^2(t)+u_N^2(t)$. We shall also consider the following assumption on $\lambda_N(0)$: \\
\newline
(\textbf{H1})  $N^{1 \over 4} \lambda_N(0) \rightarrow \bar{\lambda} \in \mathbb{R} \setminus \{0\}$ in probability, as $N \rightarrow +\infty$.
\newline
\begin{remark}\label{remark_on_initial}
 The condition $\bar{\lambda} \neq 0$ in Hypothesis (\textbf{H1}) is of pure technical nature: the change of variable in the proof of Theorem \ref{mainthm} is singular in the origin, so we require the process involved does not start in the origin. We believe this assumption could be avoided by approximation arguments that we have not succeed to complete.
%
\end{remark}

\begin{theorem} \label{mainthm}
If (\textbf{H1}) holds, for $t\in[0,T]$, as $N\to\infty$, the process $\kappa_N(t)$ converges, in sense of weak convergence of stochastic processes, to the unique solution of the stochastic differential equation \begin{equation}
\label{eqlimite} d\kappa(t)=\left(4\beta^2-\frac{\beta}{2}\kappa^2(t)\right)dt+2\beta\sqrt{2\kappa(t)}dB(t)
\end{equation} with initial condition $\kappa(0)={\beta \over \beta-1}\bar{\lambda}^2$.
\end{theorem}

\subsection{The two-population Curie-Weiss model and its critical dynamics}
Let's now briefly analyse a different spin-flip system whose limiting dynamics also presents a Hopf bifurcation: we will study the critical fluctuations and we will see that they belong to the same class of universality given by Theorem \ref{mainthm}.\\

Let $\mathcal{S}=\{-1,+1\}$ and $\underline{\sigma}=(\sigma_i)_{i=1}^N\in\mathcal{S}^N$ as above, but now we divide the spin population in two disjoint groups, $I_1$ and $I_2$, such that $|I_1|=N_1$, $|I_2|=N_2$ and $N_1+N_2=N$. Let $\gamma$ denote the proportion of particles belonging to the first group, namely $\gamma:=N_1/N$. Interaction between particles depends on the population they belong to: we have two \textit{intra-group} interactions, tuning how strongly sites in the same group feel each other and controlled by the parameters $J_{11}$ and $J_{22}$, and two \textit{inter-group} interactions, giving the magnitude of the influence between particles of distinct populations, controlled by the parameters $J_{12}$ and $J_{21}$. For any configuration of the system, we define the quantities 
\begin{equation}
\label{doublemagnetization}
m_{1,N}={1\over N}\sum_{j\in I_1} \sigma_j, \:\:\:\:\:\:\:\:\:\:\:\: m_{2,N}={1\over N}\sum_{j\in I_2} \sigma_j.
\end{equation}
Let us also introduce the following functions:
\begin{eqnarray}
\label{R1R2} \mathcal{R}_1(x,y)=J_{11}x+J_{12}y,\:\:\:\:\:\:\:\:\:\:\:\: \mathcal{R}_2(x,y)=J_{21}x+J_{22}y.
\end{eqnarray}
Now we are ready to assign the spin flip rate for the process $(\underline{\sigma}(t))_{t\in[0,T]}$: the transition $\underline{\sigma}\to\underline{\sigma}^i$ occurs at rate 
\begin{equation} \label{Rates}
\begin{cases}
e^{-\sigma_i \mathcal{R}_1(m_{1,N}(t),m_{2,N}(t))}, \:\:\:\:\:\: \:\:\:\:\:\: \textrm{if} \:\: i\in I_1,\\ 
e^{-\sigma_i  \mathcal{R}_2(m_{1,N}(t),m_{2,N}(t))}, \:\:\:\:\:\: \:\:\:\:\:\: \textrm{if} \:\: i\in I_2.
\end{cases}
\end{equation}
According to \eqref{Rates}, we are dealing with a Markov process $\underline{\sigma}(t)\in \mathcal{S}^N$ whose infinitesimal generator is 
\begin{equation}
\label{infgen2pop} L_N f(\underline{\sigma})= \left(\sum_{i\in I_1} e^{-\sigma_i \mathcal{R}_1(m_{1,N},m_{2,N})}+\sum_{i\in I_2} e^{-\sigma_i \mathcal{R}_2(m_{1,N},m_{2,N})}\right) \nabla^\sigma_i f(\underline{\sigma}),
\end{equation}
where $\nabla^\sigma_i f(\underline{\sigma})= f(\underline{\sigma}^i)-f(\underline{\sigma})$.
\begin{theorem} For $t\in[0,T]$, the process $(m_{1,N}(t),m_{2,N}(t))\in [-1,1]^2$, defined by the expression \eqref{doublemagnetization}, is an order parameter of the system.
As $N\to+\infty$ in such a way the proportion $\gamma$ remains constant, the process $(m_{1,N}(t),m_{2,N}(t))_{t\in[0,T]}$ converges, in sense of weak convergence of stochastic processes, to a limiting deterministic process, solution of the system of ordinary differential equations 
\begin{equation}
\label{limitingdyn2pop}
\begin{cases}
\dot{m_1}(t)=2\gamma \sinh \Big(\mathcal{R}_1(m_1(t),m_2(t)\Big) - 2 m_1(t) \cosh\Big(\mathcal{R}_1(m_1(t),m_2(t)\Big),\\
\dot{m_2}(t)=2(1-\gamma) \sinh \Big(\mathcal{R}_2(m_1(t),m_2(t)\Big) - 2 m_2(t) \cosh\Big(\mathcal{R}_2(m_1(t),m_2(t)\Big)
\end{cases}
\end{equation}
\end{theorem}
As pointed out in \cite{CFT16}, the parameters $\gamma, J_{11}, J_{12}, J_{21}, J_{22}$ can be adjusted to create an Hopf bifurcation at the origin by imposing that
\begin{eqnarray}
\label{condizionecritica1} \gamma J_{11} -1 = -((1-\gamma)J_{22}-1),\\
\label{condizionecritica2} \Gamma:=(\gamma J_{11}-1)^2 + \gamma (1-\gamma) J_{12}J_{21} <0,
\end{eqnarray}
In this case, the matrix obtained linearizing \eqref{limitingdyn2pop} around $(0,0)$ will be 
$$
A_{cr}=\left( \begin{array}{cc}
2(\gamma J_{11} -1) & 2\gamma J_{12} \\ 2(1-\gamma) J_{21} & -2(\gamma J_{11} -1)
\end{array}   \right) 
$$
and its eigenvalues are the purely imaginary  numbers $\lambda_{1,2}=\pm 2 i \sqrt{|\Gamma|}$.\\
Notice that a result concerning standard fluctuations similar to Theorem \ref{NormalFlucCWdiss} can be stated but we   focus on the critical fluctuations of the process $(m_{1,N}(t),m_{2,N}(t))$ when \eqref{condizionecritica1} and \eqref{condizionecritica2} hold, hence in presence of a Hopf bifurcation. The critical fluctuation flow is  described by the process 
\begin{equation}
\label{XY2pop}  x_N(t)= N^{1\over 4} m_{1,N}(N^{1\over 2}t), \:\:\:\:\:\:\:\:\:\: y_N(t)=N^{1\over 4} m_{2,N}(N^{1\over 2}t).
\end{equation}
Consider the change of variables \begin{equation}\label{WV2pop}
w_N(t)= {y_N(t)\over (1-\gamma)J_{21}}, \:\:\:\:\:\:\:\:\:\: v_N(t)= {1\over \sqrt{|\Gamma|}}\left(-x_N(t) + {(\gamma J_{11}-1)\over (1-\gamma)J_{21}} y_N(t)\right).
\end{equation}
and define $$\kappa_N(t)=w_N(t)^2+v_N(t)^2. $$ Similarly to the case with dissipation, our technique to prove the convergence for the process $(\kappa_N(t))_{t\in[0,T]}$ requires to fix initial conditions such that 
\begin{equation} \label{condizioniiniziali2pop} (m_{1,N}(0),m_{2,N}(0))\overset{w}{\underset{N\to+\infty}{\longrightarrow}} (0,0), \:\:\:\:\:\:\:\: \kappa_N(0)\overset{w}{\underset{N\to+\infty}{\longrightarrow}} \bar{\kappa}\end{equation} with $\bar{\kappa}>0$. In this case, to obtain initial conditions which verify \eqref{condizioniiniziali2pop}, one can take a small asymmetry in the initial distribution for the spins.  For simplicity, 
we introduce this small asymmetry only in one family of spins. As noticed in Remark \ref{remark_on_initial}, we believe this asymmetry could be avoided.
\newline
(\textbf{H2})  the initial spins $\{\sigma_i(0)\}_{i=1,\dots,N}$ constitute a family of independent random variable with the distributions satisfying the following conditions: \begin{itemize}
\item if $i\in I_1$, then 
\[ 
\lim_{N \rightarrow +\infty} N^{1 \over 4}\left[ P(\sigma_i(0)=+1)- {1\over 2} \right] = \epsilon \neq 0;\]
\item if $i\in I_2$, then
\[ 
\lim_{N \rightarrow +\infty} N^{1 \over 4}\left[ P(\sigma_i(0)=+1)- {1\over 2} \right] = 0;\]
\end{itemize}
If (\textbf{H2}) holds, then $x_N(0)\to 2\epsilon \gamma$ and $y_N(0)\to 0$ in distribution, hence $\kappa_N(0)$ weakly converges to $4 \epsilon^2 \gamma^2|\Gamma|^{-1}$.\\
%

\begin{theorem} \label{maintheorem2pop}
Assume (\textbf{H2}) holds. Take $\gamma\in]0,1[$, $(J_{11},J_{12},J_{21},J_{22})\in \mathbb{R}^4$ such that \eqref{condizionecritica1}-\eqref{condizionecritica2} are verified and such that $Z_2(\gamma,J_{11},J_{12},J_{21})<0$, where
\begin{align*}
Z_2(\gamma,J_{11},J_{12},J_{21})&=-2 J_{11}^2 |\Gamma| - 2 J_{21}^2 + (\gamma J_{11}-1)|\Gamma|(J_{11}^2-J_{21}^2) + (\gamma J_{11} -1) J_{21}^2 \\
&+ (\gamma J_{11}-1) (J_{11}( \gamma J_{11}-1) + (1-\gamma) J_{12} J_{21} )^2 +\\&- J_{11}(\gamma J_{11}-1)(J_{11}( \gamma J_{11}-1) + (1-\gamma) J_{12} J_{21} ) .
\end{align*}
Then, for $t\in[0,T]$, as $N\to+\infty$ in such a way $\gamma$ remains constant, the process $\kappa_N(t)$ converges, in sense of weak convergence of stochastic processes, to the unique solution of the stochastic differential equation
\begin{equation}
\label{limitfluttuazioni2pop} d\kappa(t)= \left(4Z_1(\gamma,J_{11},J_{12},J_{21})+{1\over 4} Z_2(\gamma,J_{11},J_{12},J_{21})\kappa^2(t)\right)dt+2\sqrt{2Z_1(\gamma,J_{11},J_{12},J_{21})\kappa(t)}dB(t)
\end{equation}
 with 
 $$
 Z_{1}(\gamma, J_{11}, J_{12}, J_{21})= {|\Gamma|+\gamma(1-\gamma)J^2_{21} + (\gamma J_{11}-1)^2\over (1-\gamma)J^2_{21}|\Gamma|}
 $$
 and $\kappa(0)=4 \epsilon^2 \gamma^2|\Gamma|^{-1}$.
\end{theorem}

\begin{remark}
Notice that requiring  $Z_2(\gamma,J_{11},J_{12},J_{21})<0$ guarantees global existence and uniqueness of the solution of \eqref{limitfluttuazioni2pop}. It is important to state this assumption formulating Theorem \ref{maintheorem2pop}, since it is easy to find choices for $\gamma,J_{11},J_{12},J_{21},J_{22}$ which satisfy \eqref{condizionecritica1}-\eqref{condizionecritica2} but not $Z_2(\gamma,J_{11},J_{12},J_{21})<0$: for example, one can check it with $\gamma=0.6$, $J_{11}=-10$, $J_{12}=20$, $J_{21}=-15$.
\end{remark}

\begin{remark}
The stochastic differential equations \eqref{eqlimite} and \eqref{limitfluttuazioni2pop}, which describe the limit of the critical dynamics at a Hopf bifurcation in the two models, have the same structure, i.e.
\begin{equation}
d\kappa(t)= (C_1-C_2 \kappa^2(t))dt+\sqrt{C_3 \kappa(t)}dB(t), \label{generalequation}
\end{equation}x with $C_1,C_2,C_3>0$. \end{remark}

\section{Proof of Theorem \ref{mainthm}}
Let us try to sketch the idea of the proof before going into the details: we describe the behavior of the pair $(z_N(t),u_N(t))$ through the polar coordinates $(\kappa_N(t),\theta_N(t))$ such that $$z_N(t)=\sqrt{\kappa_N(t)}\cos\theta_N(t), \:\:\:\:\:\: u_N(t)=\sqrt{\kappa_N(t)}\sin\theta_N(t).$$ The dominant part of $\theta_N$ can be approximated by a deterministic drift of order $N^{1\over 2}$, hence we get convergence of the radial variable $\kappa_N(t)$ to the process \eqref{eqlimite}  by a suitable averaging principle. Actually, this convergence will be proved through a localization argument: first of all, we analyse the convergence of the process $\kappa_N(t)$ stopped when it becomes too large or too small. By means of these stopping times (in particular the one related to the lower bound) we are able to avoid technical problems due to the singularity of the polar coordinates in the origin. Secondly, we characterize the limit of the stopped process as the solution of the \textit{stopped} martingale problem related to the infinitesimal generator of the solution of \eqref{eqlimite}. Finally, we will exploit this characterization of the limit of the stopped process to get the thesis of Theorem \ref{mainthm}.

\subsection{Preliminary computations}\label{pre}
In this subsection we perform some preliminary computations which will be useful in the following. We will also prove existence and uniqueness of the limiting process \eqref{eqlimite}.\\

\subsubsection{Change of variables} First of all, we want to give a motivation for the change of variable \eqref{u}. 
Let $\mathcal{L}_N$ be the infinitesimal generator of the process $(\hat{m}_N(t),\hat{\lambda}_N(t))$ defined in \eqref{wz}. By expanding the generator in a similar fashion to what is done in Lemma \ref{HN},
 one can check that, for a function $f$ regular enough, 
 \begin{equation}
 \label{genfastslow}
 \mathcal{L}_Nf(\hat{m},\hat{\lambda})=
 N^{1\over 2}\mathcal{L}_1f(\hat{m},\hat{\lambda})+\mathcal{L}_2f(\hat{m},\hat{\lambda})+o(1).
 \end{equation} 
 In particular, the dominant part of order $N^{1\over 2}$ is given by
  $$\mathcal{L}_1f(\hat{m},\hat{\lambda})=
  (2\hat{\lambda}-2\hat{m})\partial_{\hat{m}}f(\hat{m},\hat{\lambda})
  +(2\hat{\lambda}-2\beta \hat{m})\partial_{\hat{\lambda}} f(\hat{m},\hat{\lambda})
  =(\nabla f(\hat{m},\hat{\lambda}))^\top A\left(\begin{array}{c}
\hat{m}\\ \hat{\lambda}
\end{array}\right), $$
where $$A=\left( \begin{array}{cc}
-2 & 2 \\ -2\beta & 2
\end{array} \right).$$
Notice that $A$ corresponds to the Jacobian matrix in $(0,0)$ for system \eqref{limitingdyn} at critical point $\beta={\alpha\over 2}+1$ and its eigenvalues are $\lambda_{1,2}=\pm i 2\sqrt{\beta-1}$.
Consider an invertible matrix $C$ such that 
$$CAC^{-1}=\left(\begin{array}{cc}
0 & -2\sqrt{\beta-1} \\ 2 \sqrt{\beta-1} &0
\end{array}\right)$$ and take the change of variables $$\left(\begin{array}{c}
z\\u
\end{array}\right) = C\left(\begin{array}{c}
\hat{m}\\\hat{\lambda}
\end{array}\right).$$
Without pretending to be formal here (see Lemmas \ref{gn} and \ref{HN} for the formal computations), one gets 
$$\mathcal{L}_1f(z,u)=(\nabla f(z,u))^\top CAC^{-1} \left(\begin{array}{c}
z\\u
\end{array}\right)= -2\sqrt{\beta-1}u\partial_{z}f(z,u)+2\sqrt{\beta-1}z\partial_{u}f(z,u), $$
then, passing to the polar coordinates $\kappa=(z)^2+(u)^2$, $\theta=\arctan(u/z)$, 
$$ \mathcal{L}_1f(\kappa,\theta)=2\sqrt{\beta-1}\partial_\theta f(\kappa,\theta).$$
It follows that $\mathcal{L}_1$, which according to \eqref{genfastslow} is the ``fast'' component of the generator $\mathcal{L}_N$, involves only the derivative with respect to $\theta$, which therefore play the role of fast variable compared to the evolution of the ``radial'' variable $\kappa$. This suggests to derive the asymptotic evolution of $\kappa$ by an averaging principle.
One can easily check that a suitable choice for $C$ is given by 
$$ C=\left(\begin{array}{cc}
0 & 1 \\ {\beta\over \sqrt{\beta-1}} & -{1\over\sqrt{\beta-1}}
\end{array}\right),$$ which justifies the change of variable \eqref{u}, i.e. 
$$\begin{cases}
z= \hat{\lambda},\\u=\frac{\beta \hat{m} - \hat{\lambda}}{\sqrt{\beta-1}}.
 \end{cases}$$
 \begin{remark}
We can give a more intuitive idea on the argument for this change of variable: as stated before, one may expect that the "dominant" behaviour of $(w_N(t),z_N(t))$ should be driven by the linear system
\begin{equation}\left(\begin{array}{c} \dot{x}(t) \\ \dot{y}(t)
\end{array}\right) = A \left( \begin{array}{c}
x(t)\\y(t)
\end{array}\right) \label{linearized}.
\end{equation}
So, studying the solutions of \eqref{linearized}:
 we look for a quadratic function $$F(x,y)=(x,y)Q\left(\begin{array}{c}x\\y
\end{array}\right)$$ which is a first integral for \eqref{linearized}. Let $Q$ be a symmetric matrix and let $X(t)=(x(t),y(t))^\top$: then $F(X(t))$ is a first integral if and only if $$\frac{d}{dt}F(X(t))=0\Leftrightarrow (\dot{X}(t))^\top QX(t)+(X(t))^\top Q\dot{X}(t)=0\Leftrightarrow A^\top Q+QA=0.$$
It can be easily checked that $$Q=\left(\begin{array}{cc}\beta & -1 \\ -1 & 1
\end{array}\right)$$ satisfies $A^\top Q+QA=0$, hence, for suitable $c\in \mathbb{R}$, the equation $$(x,y)Q\left(\begin{array}{c}x\\y
\end{array}\right)=c \:\Leftrightarrow \: \beta x^2 - 2xy + y^2 = c$$ identifies an ellipse which is an orbit of the linearized system. Hence, in order to use polar coordinates, we want to transform this ellipse into a circle and this transformation is equivalent to the change of variables described above.
\end{remark} 
\subsubsection{Normal Fluctuations}
In this section, we introduce some technical lemmas and we sketch the proof for Theorem \ref{NormalFlucCWdiss}. 
We first state without proof the following simple fact concerning Markov processes and generators.
\begin{lemma}\label{cambiovariabilegeninf}
Let $(X_t)_{t\geq 0}$ be a Markov process on a metric space $E$ admitting an infinitesimal generator $L$. Let $g:E\to F$ be a function, with $F$ metric space. Assume that, for every $f:F\to \mathbb{R}$ such that $(f\circ g)\in dom(L)$, $L(f\circ g)$ is a function of $g(x)$, i.e. $L(f\circ g) = (Kf)\circ g$. Then, $(g(X_t))_{t\geq 0}$ is a Markov process with infinitesimal generator $K$, defined by $L(f\circ g) = (Kf)\circ g$.

\end{lemma}

In what follows we say that $g(x)$ is an {\em order parameter} of the Markov process $(X_t)_{t\geq 0}$.

\begin{lemma}\label{orderparameter}
For $t\in [0,T]$, the process $(m_N(t),\lambda_N(t))\in [-1,1]\times \mathbb{R}$, defined by \eqref{lambda} and \eqref{magne}, is an order parameter of the system.
\end{lemma}
\begin{proof}
Notice that $(m_N,\lambda_N)$ is an empirical average in the sense defined above since $$m_N(t)=\int_{\mathcal{S}\times\mathbb{R}} \sigma d\rho_N(t), \:\:\:\:\:\:\:\:\:\: \lambda_N(t)=\int_{\mathcal{S}\times\mathbb{R}} \lambda d\rho_N(t).$$So we are left to prove that $(m_N,\lambda_N)$ is a Markov process: it is enough to compute its infinitesimal generator $K_N$. Consider the function $g:\mathcal{S}^N\times\mathbb{R}\to[-1,+1]\times \mathbb{R}$ defined by $$g(\underline{\sigma},\lambda)=\left( \frac{1}{N}\sum_{j=1}^N \sigma_j,\lambda\right)=:(m,\lambda).$$ If, for $f:[-1,1]\times\mathbb{R}\to \mathbb{R}$, $f\in dom(L_N)$, we compute $L_N(f\circ g)$: \begin{align*}L_N(f\circ g)(\underline{\sigma},\lambda)&=\sum_{i=1}^N \left[ \Big( 1-\tanh(\sigma_i\lambda)\Big) \left( f\left( m - \frac{2\sigma_i}{N}, \lambda-\frac{2\beta\sigma_i}{N}\right)-f\left(m,\lambda\right) \right)  \right] -\alpha\lambda f_\lambda(m,\lambda)=\\&= \sum_{j\in\mathcal{S}}\left[ |A_N(j)|\Big(1-\tanh(j\lambda)\Big) \left( f \left( m -\frac{2j}{N},\lambda-\frac{2\beta j}{N}\right)-f\left(m,\lambda\right)\right)  \right]-\alpha \lambda f_\lambda(m,\lambda) \end{align*}
where $A_N(j)$ is the set of $\sigma_i$, $i=1,\dots,N$, such that $\sigma_i=j$ with $j\in\mathcal{S}$. Then, we have \begin{equation}
\label{AN} |A_N(j)|=\frac{N\left(1+j\frac{\sum_{k=1}^N\sigma_k}{N}\right)}{2}=\frac{N(1+jm)}{2}.
\end{equation}
Therefore, thanks to Lemma \ref{cambiovariabilegeninf}, $(m_N,\lambda_N)$ is a Markov process with infinitesimal generator $K_N$ given by: \begin{equation}
\label{geninfmagne} K_N f(m,\lambda)=\sum_{j\in\mathcal{S}}\left[ |A_N(j)|\Big(1-\tanh(j\lambda)\Big) \left( f \left( m -\frac{2j}{N},\lambda-\frac{2\beta j}{N}\right)-f\left(m,\lambda\right)\right)  \right]-\alpha \lambda f_\lambda(m,\lambda).
\end{equation}
\end{proof}
\begin{proof}[Proof of Theorem \ref{NormalFlucCWdiss}]
Let $I_t$ be the infinitesimal generator of the solution of \ref{CLT}, namely, for $f\in dom(K)$:
 \begin{align*} I_t f(\tilde{m},\tilde{\lambda})&=(1-m(t)\tanh(\lambda(t))(2f_{\tilde{m}\tilde{m}}+2\beta^2f_{\tilde{\lambda}\tilde{\lambda}}+4\beta f_{\tilde{m}\tilde{\lambda}})+\\
 &+2((1+\tanh(\lambda(t))\tilde{\lambda}-\tilde{m})f_{\tilde{m}}+((2\beta-\alpha)\tilde{\lambda}-2\beta(1-\tanh(\lambda(t))\tilde{m})f_{\tilde{\lambda}}.\end{align*}
Let $K_N$ be the infinitesimal generator of the process $(m_N(t),\lambda_N(t))$ (see \eqref{geninfmagne})
and consider the time-dependent, linear and invertible tranformation 
$$
(\tilde{m},\tilde{\lambda})= g_t(m,\lambda)= (N^{1\over 2}(m-m(t)),N^{1\over 2}(\lambda-\lambda(t)):
$$ applying $K_N$ to a function $f(g_t(m,\lambda))$, by simple computations and Lemma \ref{cambiovariabilegeninf} we can check that $\tilde{I}_{t,N}$, the infinitesimal generator of the time-inhomogeneous Markov process $(\tilde{m}_N(t),\tilde{\lambda}_N(t))$ reads: \begin{align*}
\tilde{I}_{t,N} f(\tilde{m},\tilde{\lambda})&=K_N f(g_t(m,\lambda))=\\
&=\frac{N(1+m(t))+N^{1\over 2}\tilde{m}}{2}\Big(1-\tanh(N^{-{1\over 2}}\tilde{\lambda}+\lambda(t))\Big) \left( f \left( \tilde{m}
  -\frac{2}{N^{1\over 2}},\tilde{\lambda}-\frac{2\beta}{N^{1\over 2}}\right)
  -f\left(\tilde{m},\tilde{\lambda} \right)\right)  +\\
 &+\frac{N(1-m(t))-N^{1\over 2}\tilde{m}}{2}\Big(1+\tanh(N^{-{1\over 2}}\tilde{\lambda}+\lambda(t))\Big) \left( f \left( \tilde{m} +\frac{2}{N^{1\over 2}},\tilde{\lambda}+\frac{2\beta }{N^{1\over 2}}\right)-f\left(\tilde{m},\tilde{\lambda}\right)\right)+\\
 &-\alpha( \tilde{\lambda}+N^{1\over 2}\lambda(t)) f_{\tilde{\lambda}}(\tilde{m},\tilde{\lambda})-N^{1\over 2}\dot{m}(t)f_{\tilde{m}}-N^{1\over 2} \dot{\lambda}(t)f_{\tilde{\lambda}}.
\end{align*}
By Corollary 8.6, Chapter 4 of \cite{EK86}, it is enough to show that, for any $f\in\mathcal{C}^3_b( \mathbb{R}^2)$ it holds that $$\lim_{N\to\infty}\sup_{(\tilde{m},\tilde{\lambda})\in\mathbb{R}^2}|\tilde{I}_{t,N}f(\tilde{m},\tilde{\lambda})-I_t f(\tilde{m},\tilde{\lambda})|=0,$$ which can be easily checked performing a first order Taylor expansion of $\tanh(N^{-{1\over 2}}\tilde{\lambda}+\lambda(t))$ around $\lambda(t)$ and a second order Taylor expansion of $f$ around $(\tilde{m},\tilde{\lambda})$ in the expression for $\tilde{I}_{N,t}f(\tilde{m},\tilde{\lambda})$.
\end{proof}
\subsubsection{Expansion of infinitesimal generators}\label{espansionegeneratori}
In this paragraph,  we study the asymptotic expansion of the infinitesimal generators of the processes involved in the proof of the main result.
\begin{lemma}\label{gn}
In the critical case $\beta={\alpha\over 2}+1$, the infinitesimal generator $G_N$ of the Markov pair $(z_N(t),u_N(t))$ is given by: \begin{align*}
G_Nf(z,u)&=N^{1\over 2}\frac{N (1+N^{-{1\over 4}}\beta^{-1}(\sqrt{\beta-1}u+z))}{2}\Big( 1-\tanh(N^{-{1\over 4}}z)    \Big)\left( f\left( z-\frac{2\beta}{N^{3\over 4}},u\right) -f(z,u)\right)+\\ &+N^{1\over 2}\frac{N (1-N^{-{1\over 4}}\beta^{-1}(\sqrt{\beta-1}u+z))}{2}\Big( 1+\tanh(N^{-{1\over 4}}z)    \Big)\left( f\left( z+\frac{2\beta}{N^{3\over 4}},u\right) -f(z,u)\right)+\\ &+ N^{1\over 2} \Big( -2(\beta-1)zf_z(z,u)+ 2\sqrt{\beta-1}zf_u(z,u)  \Big).
\end{align*}
\end{lemma}
\begin{proof}
Define the following processes, for $t\in[0,T]$: $$\tilde{w}_N(t)=N^{1\over 4}m_N(t), \:\:\:\:\: \tilde{z}_N(t)=N^{1\over 4}\lambda_N(t), \:\:\:\:\: \tilde{u}_N(t)=\frac{\beta\tilde{w}_N(t)-\tilde{z}_N(t)}{\sqrt{\beta-1}}. $$ We want to identify the infinitesimal generator of $(\tilde{z}_N(t),\tilde{u}_N(t))$ applying Lemma \ref{cambiovariabilegeninf}: consider the function $$g(m,\lambda)=\left( N^{1\over 4}\lambda, N^{1\over 4}\frac{\beta m - \lambda}{\sqrt{\beta-1}}  \right)$$ 
and evaluate the infinitesimal generator $K_N$ defined in \eqref{geninfmagne} on the function $(f\circ g)$. It's just a matter of computation to notice that we get the infinitesimal generator $\tilde{G}_N$ defined by \begin{align*}
\tilde{G}_Nf(\tilde{z},\tilde{u})&=K_N(f\circ g)(m,\lambda)=
\\ &+  \frac{N (1+N^{-{1\over 4}}\beta^{-1}(\sqrt{\beta-1}\tilde{u}+\tilde{z}))}{2}\Big( 1-\tanh(N^{-{1\over 4}}\tilde{z})    \Big)\left( f\left( \tilde{z}-\frac{2\beta}{N^{3\over 4}},\tilde{u}\right) -f(\tilde{z},\tilde{u})\right)+
\\ &+\frac{N (1-N^{-{1\over 4}}\beta^{-1}(\sqrt{\beta-1}\tilde{u}+\tilde{z}))}{2}\Big( 1+\tanh(N^{-{1\over 4}}\tilde{z})    \Big)\left( f\left( \tilde{z}+\frac{2\beta}{N^{3\over 4}},\tilde{u}\right) -f(\tilde{z},\tilde{u})\right)+
\\ &+ \Big( -2(\beta-1)\tilde{z}f_{\tilde{z}}(\tilde{z},\tilde{u})+ 2\sqrt{\beta-1}\tilde{z}f_{\tilde{u}}(\tilde{z},\tilde{u})  \Big) .
\end{align*}
Finally, we obtain $(z_N(t),u_N(t))$ from $(\tilde{z}_N(t),\tilde{u}_N(t))$ rescaling time by a factor $N^{1\over 2}$: the infinitesimal generator $G_N$ of $(z_N(t),u_N(t))$ is obtained by $$G_Nf(z,u)=N^{1\over 2}\tilde{G}_Nf(\tilde{z},\tilde{u}).$$
\end{proof}

\begin{lemma}\label{HN}
Fix any $r,R>0$ such that $r<{\beta \over \beta-1}\bar{\lambda}^2<R$ and let $$\tau^N_{r,R}:=\inf\{t\in [0,T]\:|\: z_N^2(t)+u_N^2(t) \not\in\: ]r,R[ \}.$$ For any $t\in[0,T]$ define the polar coordinates process $(\kappa_N(t),\theta_N(t))$ by $$\kappa_N(t)=z_N^2(t)+u_N^2(t),$$  $$\theta_N(t)= \arctan\left(\frac{u_N(t)}{z_N(t)}\right). $$ Then, the stopped process $(\kappa_N(t\wedge\tau^N_{r,R}),\theta_N(t\wedge\tau^N_{r,R}))$ has an infinitesimal generator $H_N^{r,R}$ which, for a function $f:[r,R]\times[-{\pi\over 2},{\pi \over 2}]\to \mathbb{R}$, $f\in\mathcal{C}^3$, satisfies \begin{align*}
H_N^{r,R}f(\kappa,\theta)&=\mathbbm{1}_{]r,R[}(\kappa)\left[8\beta^2\kappa\cos^2\theta f_{\kappa\kappa}(\kappa,\theta) -8\beta^2\cos\theta\sin\theta f_{\kappa\theta}(\kappa,\theta)+ \frac{2\beta^2\sin^2\theta}{\kappa} f_{\theta\theta}(\kappa,\theta)+\right.\\&+\left( N^{1\over 2}2\sqrt{\beta-1} + \frac{4\beta^2\cos\theta\sin\theta}{\kappa}+\frac{2\beta}{3} \kappa\cos^3\theta\sin\theta +2\sqrt{\beta-1}\sin\theta\right)f_\theta(\kappa,\theta)+\\&+\left. \left(4\beta^2-\frac{4\beta}{3}\kappa^2\cos^4\theta\right)f_\kappa(\kappa,\theta)+o_{r,R}(1)\right], 
\end{align*}
where the remainders $o_{r,R}(1)$ can be uniformly dominated by a term of order $o(1)$ on $[r,R]\times[-{\pi\over 2}, {\pi\over 2}]$.
\end{lemma}
\begin{remark}\label{positivityofkappa}
The process $\theta_N(t)$ defined in Lemma \ref{HN} is almost surely well-defined for $t\in[0,T]$. In fact we have that $$P\Big( \exists\:\: t\in[0,T] \:\: s.t. \:\: z_N(t)=0\Big)=0:$$ from \eqref{lambda}, we can see that $\lambda_N$ (hence $z_N$) can hit 0 only when a jump occurs. Let $(\tau_n)_n$ be a sequence of jump times for $\lambda_N$: for any $n$, we have that $$P\big( \lambda_N(\tau_n)=0\:\big|\:\mathcal{F}_{\tau_{n-1}} \big)\leq P\big(	\tau_n\in A_{n-1}\:\big|\:\mathcal{F}_{\tau_{n-1}}	\big)$$ where $A_{n-1}$ is an aleatory set such that $|A_{n-1}|\leq 1$. In fact, intuitively, since the jump size of $\lambda_N$ is fixed and the trajectories of $\lambda_N$ are strictly increasing or strictly decreasing between two jumps, then there exists at most one point in which the jump leading to 0 can occur. Moreover, the distribution of the jump times is absolutely continuous with respect to the Lebesgue measure over $\mathbb{R}$, hence $$P\big(	\tau_n\in A_{n-1}\:\big|\:\mathcal{F}_{\tau_{n-1}}	\big)=0.$$ In conclusion, $$P\Big( \exists\:\: t\in[0,T] \:\: s.t. \:\: z_N(t)=0\Big)\leq \sum_n E\left[ P\big(	\tau_n\in A_{n-1}\:\big|\:\mathcal{F}_{\tau_{n-1}}	\big)\right]=0.$$
\end{remark}
\begin{proof}
Consider the function $$g(z,u)=\left(z^2+u^2,\arctan\left({u\over z}\right)\right)$$ and apply generator $G_N$ defined in Lemma \ref{gn} to $(f\circ g)(z,u)$ with $f\in\mathcal{C}^3([r,R]\times[-{\pi\over 2}, {\pi\over 2}])$. By standard Taylor expansions we get: \begin{align*}
&G_N(f\circ g)(z,u)=\\ &=\frac{N^{3\over 2} (\beta+N^{-{1\over 4}}(\sqrt{\beta-1}u+z))}{2\beta}\Big( 1-\tanh(N^{-{1\over 4}}z)    \Big)\left( (f\circ g)\left( z-\frac{2\beta}{N^{3\over 4}},u\right) -(f\circ g)(z,u)\right)+\\ &+\frac{N^{3\over 2} (\beta-N^{-{1\over 4}}(\sqrt{\beta-1}u+z))}{2\beta}\Big( 1+\tanh(N^{-{1\over 4}}z)    \Big)\left( (f\circ g)\left( z+\frac{2\beta}{N^{3\over 4}},u\right) -(f\circ g)(z,u)\right)+\\
 &+ N^{1\over 2} \Big( -2(\beta-1)z(f\circ g)_z+ 2\sqrt{\beta-1}z(f\circ g)_u  \Big)=\\&=\frac{\beta N^{3\over 2} +N^{5\over 4}(\sqrt{\beta-1} u +z)}{2\beta}   \left( 1-\frac{z}{N^{1\over 4}}+\frac{z^3}{3N^{3\over 4}}+o\left(\frac{1}{N^{3\over 4}}\right)\right)\left( -\frac{2\beta}{N^{3\over 4}}(f\circ g)_z + \frac{2\beta^2}{N^{3\over 2}}(f\circ g)_{zz} + o\left(\frac{1}{N^{3\over 2}}\right)\right)+\\ &+  
 \frac{\beta N^{3\over 2} -N^{5\over 4}(\sqrt{\beta-1} u +z)}{2\beta}   \left( 1+\frac{z}{N^{1\over 4}}-\frac{z^3}{3N^{3\over 4}}+o\left(\frac{1}{N^{3\over 4}}\right)\right)\left( \frac{2\beta}{N^{3\over 4}}(f\circ g)_z - \frac{2\beta^2}{N^{3\over 2}}(f\circ g)_{zz} + o\left(\frac{1}{N^{3\over 2}}\right)\right)+\\ 
 &+ N^{1\over 2} \Big( -2(\beta-1)z(f\circ g)_z+ 2\sqrt{\beta-1}z(f\circ g)_u  \Big) =\\&
 =2\beta^2(f\circ g)_{zz} - \left( \frac{2\beta z^3}{3} + N^{1\over 2} 2\beta \sqrt{\beta-1}u \right)(f\circ g)_z + N^{1\over 2} 2\sqrt{\beta-1}z(f\circ g)_u + o(1).
\end{align*}
Now, observe that: $$(f\circ g)_z=2z f_\kappa - \frac{u}{z^2+u^2} f_\theta,\:\:\:\:\:\:\:\:\: (f\circ g)_u=2uf_\kappa + \frac{z}{z^2+u^2}f_\theta,$$
$$ (f\circ g)_{zz}= 4z^2f_{\kappa\kappa} - \frac{4zu}{z^2+u^2}f_{\kappa\theta} + \frac{u^2}{(z^2+u^2)^2}f_{\theta\theta} + 2f_\kappa + \frac{2uz}{(z^2+u^2)^2}f_\theta,$$
hence \begin{align*}
G_N(f\circ g)(z,u)&=\:8\beta^2 z^2 f_{\kappa\kappa} - \frac{8\beta^2zu}{z^2+u^2}f_{\kappa\theta} +\frac{2\beta^2u^2}{(z^2+u^2)^2}f_{\theta\theta}+ \\ &+\left( N^{1\over 2} 2\sqrt{\beta-1} + \frac{4\beta^2zu}{(z^2+u^2)^2}+\frac{2\beta z^3u}{3(z^2+u^2)} \right)f_\theta+ \\&+ \left( 4\beta^2 -  \frac{4\beta z^4}{3}\right)f_\kappa + o(1),
\end{align*}
which finally yields, thanks to Lemma \ref{cambiovariabilegeninf}, that the infinitesimal generator $H_N$ of $(\kappa_N(t),\theta_N(t))$ satisfies: \begin{align*}
H_Nf(\kappa,\theta)&=8\beta^2\kappa\cos^2\theta f_{\kappa\kappa}(\kappa,\theta) -8\beta^2\cos\theta\sin\theta f_{\kappa\theta}(\kappa,\theta)+ \frac{2\beta^2\sin^2\theta}{\kappa} f_{\theta\theta}(\kappa,\theta)+\\&+\left( N^{1\over 2}2\sqrt{\beta-1} + \frac{4\beta^2\cos\theta\sin\theta}{\kappa}+\frac{2\beta}{3} \kappa\cos^3\theta\sin\theta +2\sqrt{\beta-1}\sin\theta\right)f_\theta(\kappa,\theta)+\\&+ \left(4\beta^2-\frac{4\beta}{3}\kappa^2\cos^4\theta\right)f_\kappa(\kappa,\theta)+o(1).
\end{align*}
Therefore, the infinitesimal generator $H^{r,R}_N$ of the stopped process $(\kappa_N(t\wedge\tau^N_{r,R}),\theta(t\wedge\tau^N_{r,R})$ satisfies, for a function $f\in\mathcal{C}^3([r,R]\times[-{\pi\over 2},{\pi \over 2}])$: \begin{align*}
H_N^{r,R}f(\kappa,\theta)&=\mathbbm{1}_{]r,R[}(\kappa)\left[8\beta^2\kappa\cos^2\theta f_{\kappa\kappa}(\kappa,\theta) -8\beta^2\cos\theta\sin\theta f_{\kappa\theta}(\kappa,\theta)+ \frac{2\beta^2\sin^2\theta}{\kappa} f_{\theta\theta}(\kappa,\theta)+\right.\\&+\left( N^{1\over 2}2\sqrt{\beta-1} + \frac{4\beta^2\cos\theta\sin\theta}{\kappa}+\frac{2\beta}{3} \kappa\cos^3\theta\sin\theta +2\sqrt{\beta-1}\sin\theta\right)f_\theta(\kappa,\theta)+\\&+\left. \left(4\beta^2-\frac{4\beta}{3}\kappa^2\cos^4\theta\right)f_\kappa(\kappa,\theta)+o_{r,R}(1)\right], 
\end{align*}
Notice that the remainders of the expansion of $\tanh$ and $(f\circ g)$ are continuous functions of $(\kappa,\theta)$ on the compact set $[r,R]\times[-{\pi\over 2},{\pi\over 2}]$, so $o_{r,R}(1)$ can be uniformly dominated by a term of order $o(1)$. 
\end{proof}

\subsubsection{Existence and uniqueness of limit process}
In this paragraph we prove the well-posedness of the SDE \eqref{eqlimite}. 
\begin{lemma}\label{existenceanduni}
There exists a unique solution of the limiting equation \eqref{eqlimite}.
\end{lemma}

\begin{proof}
Let $(X(t),Y(t))_{t\geq 0}$ be the unique and globally existent solution of \begin{equation}
\label{xy} \begin{cases} dX(t)=-\frac{\beta^3}{2}X(t)(X^2(t)+Y^2(t))dt+dB_1(t) \\ dY(t)=-\frac{\beta^3}{2}Y(t)(X^2(t)+Y^2(t))dt+dB_2(t)  \end{cases}
\end{equation} with $B_1,B_2$ independent Brownian motion and $X(0)=Y(0)={\bar{\lambda}\over 2\sqrt{\beta(\beta-1)}}$. Now, let us define \begin{equation}
\label{ksolution}\kappa(t\wedge\tau_{r,R})=2\beta^2(X^2(t\wedge\tau_{r,R})+Y^2(t\wedge\tau_{r,R}))
\end{equation} with $$\tau_{r,R}:=\inf_{t\in[0,T]} \{t\geq 0 \:|\: 2\beta^2(X^2(t)+Y^2(t))\not\in \: ],r,R[\}.$$
Applying Ito's formula to \eqref{ksolution} for $t\in[0,T\wedge\tau_{r,R}]$: $$d\kappa(t)=\left(4\beta^2-\frac{\beta}{2}\kappa^2(t)\right)dt+2\beta\sqrt{2\kappa(t)}\left( \frac{X(t)}{\sqrt{X^2(t)+Y^2(t)}}dB_1(t)+\frac{Y(t)}{\sqrt{X^2(t)+Y^2(t)}}dB_2(t)  \right).$$
Notice that, for $t\in[0,T\wedge\tau_{r,R}]$, the process $$\int_0^t \frac{X(s)}{\sqrt{X^2(s)+Y^2(s)}}dB_1(s)+\frac{Y(s)}{\sqrt{X^2(s)+Y^2(s)}}dB_2(s)$$ is a zero-mean martingale with quadratic variation $dt$, hence it is a Brownian motion. Therefore, the process $\kappa(t)$ defined in \eqref{ksolution}, is a solution of \eqref{eqlimite} on the time interval $[0,T\wedge\tau_{r,R}]$ and existence is proved. On the same time interval, also uniqueness holds for the solution of \eqref{eqlimite} since the coefficients $$b(k)=4\beta^2-\frac{\beta}{2}\kappa, \:\:\:\:\:\: \sigma(\kappa)=2\beta\sqrt{2\kappa}$$ are Lipschitz-continuous functions over $[r,R]$.\\
To conclude the proof, we want to show that \begin{equation} \label{zeroprobabilitystoppingtimes} P(\tau_{r,R}\leq T)\to 0, \:\:\:\: as\:\:r\to0^+,R\to+\infty.\end{equation}
By the global existence of $(X(t),Y(t))_{t\geq 0}$,
 $$P\left(\sup_{t\in[0,T]} 2\beta^2(X^2(t)+Y^2(t))\geq R\right)\to 0$$
  as $R\to \infty.$ On the other hand, take a sequence $(r_n)_{n\geq 1}$ of positive numbers converging monotonically to zero. For any $n\geq 1$, define the event $A_n$ as $$A_n:=\left\{\inf_{t\in[0,T]} 2\beta^2(X^2(t)+Y^2(t))\leq r_n \right\} $$and notice that $(A_n)_{n\geq 1}$ is a decreasing sequence of events converging to $$\bar{A}:=\bigcap_n A_n = \left\{ \inf_{t\in[0,T]} 2\beta^2(X^2(t)+Y^2(t))=0 \right\}=\Big\{ \: \exists\: t\in[0,T] \: \textrm{s.t.}\: X(t)=Y(t)=0\Big\}.$$  Finally, recall that $(X(t),Y(t))$, being a bidimensional diffusion, is absolutely continuous with respect to a bidimensional Brownian motion: it's easy to see that a bidimensional Brownian motion never visits the origin a.s., and we conclude by $$P\Big( \: \exists\: t\in[0,T] \: \textrm{s.t.}\: X(t)=Y(t)=0\Big)=0.$$
\end{proof}

\subsection{Tightness of the stopped process}

In this section, we want to prove that the sequence of stopped processes $(\kappa_N(t\wedge\tau^N_{r,R}))_{N\geq 1}$ defined in Lemma \ref{HN} is tight.

\begin{remark}\label{stoppedmeasure}
Let $P_N$ be the law of $(\kappa_N(t))_{t\in[0,T]}$ on $\mathcal{D}([0,T],\mathbb{R})$, endowed with the Skorohod topology. For any $x\in \mathcal{D}([0,T],\mathbb{R})$ and any $r,R>0$ consider $$\tau_{r,R}:=\inf\{t\in[0,T]\:|\: x(t)\not\in\: ]r,R[\}$$ and define the map $\varphi_{r,R}:\mathcal{D}([0,T],\mathbb{R})\to\mathcal{D}([0,T],\mathbb{R})$ as \begin{equation} \label{phi}\varphi_{r,R}\left(x(\cdot)\right)=x(\cdot\wedge \tau_{r,R}).\end{equation}
Formally speaking, "the sequence of stopped processes $(\kappa_N(t\wedge\tau^N_{r,R}))_{N\geq 1}$  is tight" means "the sequence of probability measures $(P_N\circ\varphi^{-1}_{r,R})_{N\geq 1}$ is tight".
\end{remark}

\begin{proposition}\label{tight}
For any $N\geq 1$, $t\in [0,T]$ let $\kappa_N(t\wedge\tau^N_{r,R})$ be the process defined in Lemma \ref{HN}. Then, the sequence of processes $(\kappa_N(t\wedge\tau^N_{r,R}))_{N\geq 1}$ is tight.
\end{proposition}
\begin{proof}
First of all, for $j\in\mathcal{S}$, $t>0$, consider the set $A_N(j,t)$ of spins equal to $j$ at time $t$ (see also \eqref{AN}): we have \begin{equation}
|A_N(j,t)|=\frac{N(1+jm_N(t))}{2}.
\end{equation}
We can write the jump part of the process $z_N(t)$ in the following way, for $t\in[0,T]$: \begin{equation}
\label{zjump} \int_0^t \sum_{j\in\mathcal{S}} \left[  - \frac{j2\beta }{N^{3\over 4}} \right]\Lambda_N(j,ds),
\end{equation} where $\Lambda(j,ds)$ indicates a point process of intensity $$N^{1\over 2} |A_N(j,N^{1\over 2}s)|(1-j\tanh\lambda_N(N^{1\over 2}s)) dt.$$ From \eqref{lambda} and \eqref{u}, it is easy to see that $u_N(t)$ has continuous trajectories, so for any $t\in[0,T]$, $$\kappa_N(t)-\kappa_N(t-)=z_N^2(t)-z_N^2(t-).$$ 
Let us study the jumps of the process $z_N^2(t)$: Using the generalized It\^{o}'s formula (see \cite{IW81}),  we can write the jump part of $z_N^2(t)$ as
 \begin{equation}
\label{kjump}\int_0^t\sum_{j\in\mathcal{S}} \left[ \left(z_N(s) - \frac{j2\beta }{N^{3\over 4}}\right)^2-(z_N(s))^2 \right]\Lambda_N(j,dt)=\int_0^t \sum_{j\in\mathcal{S}} \left[ \frac{4\beta^2}{N^{3\over 2}} - \frac{j2\beta z_N(s)}{N^{3\over 4}} \right]\Lambda_N(j,dt).
\end{equation}
The stopped process $\kappa_N(t\wedge\tau^N_{r,R})$ can be decomposed in the following way: \begin{equation}
\label{decomposedK} \kappa_N(t\wedge\tau^N_{r,R})=\kappa_N(0)+\int_0^{t\wedge\tau^N_{r,R}}H_N^{r,R} \kappa_N(s) ds + \mathcal{M}^{t\wedge\tau^N_{r,R}}_{N,\kappa}
\end{equation}
where $H_N^{r,R}\kappa$ indicates the infinitesimal generator of Lemma \ref{HN} evaluated on the function $f(\kappa,\theta)=\kappa$, while $\mathcal{M}_{N,\kappa}$ is the local martingale given by \begin{equation}
\label{localmart} \int_0^{t\wedge\tau^N_{r,R}} \sum_{j\in\mathcal{S}}\left[ \frac{4\beta^2}{N^{3\over 2}} - \frac{j2\beta z_N(s)}{N^{3\over 4}} \right]\tilde{\Lambda}_N(j,ds).
\end{equation}
We will use the \textit{Aldous' tightness criterion} (see \cite{CE88}): a sequence of processes $\{\xi_N(t)\}_{N\geq 1}$ is tight if: \begin{enumerate}
\item for every $\varepsilon>0$ there exists $M>0$ such that \begin{equation}
\label{tight1} \sup_N P\left(\sup_{t\in[0,T]}|\xi_N(t)|\geq M\right)\leq \varepsilon,
\end{equation}
\item for every $\varepsilon>0$ and $\alpha>0$ there exists $\delta>0$ such that \begin{equation}
\label{tight2} \sup_N\sup_{0\leq\tau_1\leq\tau_2\leq(\tau_1+\delta)\wedge T} P\left( |\xi_N(\tau_2)-\xi_N(\tau_1)|\geq \alpha \right) \leq \varepsilon,
\end{equation} where the second sup is taken over stopping times $\tau_1$ and $\tau_2$, adapted to the filtration generated by process $\xi_N$.
\end{enumerate}
Notice that, for any $t\in[0,T]$, $|\kappa_N(t)|=\kappa_N(t)$ and $\kappa_N(t\wedge\tau^N_{r,R})\leq R$ hence \eqref{tight1} trivially holds.
Let now $\tau_1$, $\tau_2$ be two stopping times adapted to the the filtration generated by $\kappa_N$, such that, for $\delta>0$, $\tau_1\leq \tau_2\leq (\tau_1+\delta)\wedge T$ a.s.. By decomposition \eqref{decomposedK}, we have that 
$$\left|\kappa_N(\tau_2\wedge\tau^N_{r,R}) - \kappa_N(\tau_1\wedge\tau^N_{r,R})\right|=\left| \int_{\tau_1\wedge\tau^N_{r,R}}^{\tau_2\wedge\tau^N_{r,R}}H_N^{r,R}\kappa_N(s)ds + \mathcal{M}^{(\tau_1,\tau_2)\wedge\tau^N_{r,R}}_{N,\kappa} \right|\leq $$ $$\leq \int_{\tau_1\wedge\tau^N_{r,R}}^{\tau_2\wedge\tau^N_{r,R}} \left| H_N^{r,R}\kappa_N(s)\right|ds + \left| \mathcal{M}^{(\tau_1,\tau_2)\wedge\tau^N_{r,R}}_{N,\kappa} \right|$$ 
where
 $$\mathcal{M}^{(\tau_1,\tau_2)\wedge\tau^N_{r,R}}_{N,\kappa}:= \int_{\tau_1\wedge\tau^N_{r,R}}^{\tau_2\wedge\tau^N_{r,R}} \sum_{j\in\mathcal{S}} \left[ \frac{4\beta^2}{N^{3\over 2}} -\frac{j2\beta z_N(s)}{N^{3\over 4}}\right]\tilde{\Lambda}_N(j,ds).$$
 Let's give some estimations:
if $f(\kappa,\theta)=k$, by Lemma \ref{HN} $$H_N^{r,R}f(\kappa,\theta)=\mathbbm{1}_{]r,R[}(\kappa)\left[4\beta^2-\frac{4\beta}{3}\kappa^2\cos^4\theta + o_{r,R}(1)\right]$$  so
 $$ \int_{\tau_1\wedge\tau^N_{r,R}}^{\tau_2\wedge\tau^N_{r,R}} |H_N^{r,R}\kappa_N(s)|ds=\int_{\tau_1\wedge\tau^N_{r,R}}^{\tau_2\wedge\tau^N_{r,R}} \left|4\beta^2-\frac{4\beta}{3}(\kappa_N(s))^2\cos^4\theta_N(s) + o_{r,R}(1)\right|ds\leq$$ \begin{equation*}
 \leq \int_{\tau_1\wedge\tau^N_{r,R}}^{\tau_2\wedge\tau^N_{r,R}} 4\beta^2+{4\beta\over 3}R^2 + o_{r,R}(1)ds\leq \int_{\tau_1\wedge\tau^N_{r,R}}^{\tau_2\wedge\tau^N_{r,R}} 4\beta^2+{4\beta\over 3}R^2+ \frac{C_1(R)}{N^{\gamma}}ds,
\end{equation*} where $\gamma>0$ and $C_1(R)$ is a constant which depends on $R$. Then, taking $C_2(R):=4\beta^2+{4\beta\over 3}R^2+ C_1(R)$, 
$$ \int_{\tau_1\wedge\tau^N_{r,R}}^{\tau_2\wedge\tau^N_{r,R}} |H_N^{r,R}\kappa_N(s)|ds\leq C_2(R)\delta;$$
so, given $\alpha>0$ and $\delta< {\alpha\over C_2(R)}$ \begin{equation}
 \label{tau1tau2}P\left( \int_{\tau_1\wedge\tau^N_{r,R}}^{\tau_2\wedge\tau^N_{r,R}} |H_N^{r,R}\kappa_N(s)|ds\geq \alpha \right)=0.
 \end{equation}
Moreover, if we fix $t\in[0,T]$,  we can study the expected value of $\left(\mathcal{M}^{t\wedge\tau^N_{r,R}}_{N,\kappa}\right)^2$ applying the It\^{o} isometry for stochastic integrals with respect to point processes (see again \cite{IW81}): 
$$E\left[\left(\mathcal{M}^{t\wedge\tau^N_{r,R}}_{N,\kappa}\right)^2\right]=E\left[\int_0^{t\wedge\tau^N_{r,R}}\sum_{j\in\mathcal{S}}\left[\frac{4\beta^2}{N^{3\over 2}}-\frac{j2\beta z_N(s)}{N^{3\over 4}}\right]^2 N^{1\over 2}|A_N(j,N^{1\over 2}s)|(1-j\tanh\lambda_N(N^{1\over 2}s))ds\right]\leq$$ 
$$
\leq E\left[ \int_0^{t\wedge\tau^N_{r,R}} \sum_{j\in\mathcal{S}}  2\left[ \frac{16\beta^4}{N^3} + \frac{4\beta^2z_N^2(s)}{N^{3\over 2}}\right] N^{1\over 2} |A_N(j,N^{1\over 2}s)|2ds\right]\leq$$
$$\leq E\left[\int_0^{t\wedge\tau^N_{r,R}} 8 \left[ \frac{16\beta^4}{N^3}+\frac{4\beta^2 R}{N^{3\over 2}}\right] N^{1\over 2} \sup_{j\in\mathcal{S}} |A_N(j,N^{1\over 2}s)|ds\right]\leq E\left[ \int_0^{t\wedge\tau^N_{r,R}} 32\beta^2\left[\frac{4\beta^2}{N^{3\over 2}} + R\right] ds  \right].
$$
By the Optional Sampling Theorem, 
$$
 E\left[\left(\mathcal{M}^{(\tau_1,\tau_2)\wedge\tau^N_{r,R}}_{N,\kappa}\right)^2\right]=E\left[\left(\mathcal{M}^{\tau_2\wedge\tau^N_{r,R}}_{N,\kappa}\right)^2-\left(\mathcal{M}^{\tau_1\wedge\tau^N_{r,R}}_{N,\kappa}\right)^2\right]\leq  $$
  $$ \leq E\Big[ 32\beta^2(4\beta^2+R)\left[ (\tau_2\wedge\tau^N_{r,R})-(\tau_1\wedge\tau^N_{r,R})  \right]  \Big] \leq 32\beta^2(4\beta^2+R)\delta=: C_3(R)\delta,$$ which implies, by Chebyscev Inequality, that, given $\varepsilon>0$, $\alpha>0$ and $\delta<{\varepsilon\alpha^2\over C_3(R)  }$, \begin{equation}
\label{tau1tau2m} P\left( \left| \mathcal{M}_{N,\kappa}^{(\tau_1,\tau_2)\wedge\tau^N_{r,R}} \right|  \geq \alpha\right)\leq 		\frac{E\left[\left(\mathcal{M}_{N,\kappa}^{(\tau_1,\tau_2)\wedge\tau^N_{r,R}}\right)^2\right]}{ \alpha^2}	\leq {C_3(R)\delta\over \alpha^2}<\varepsilon.
\end{equation}
In conclusion, by \eqref{tau1tau2} and \eqref{tau1tau2m}, given $\varepsilon>0$ and $\alpha>0$, if we take $\delta<\min\left\{ {\alpha\over C_2(R)},{\varepsilon\alpha^2\over C_2(R)}\right\}$,
$$\sup_N \sup_{0\leq\tau_1\leq\tau_2\leq (\tau_1+\delta)\wedge T} P\Big(|\kappa_N(\tau_2\wedge\tau^N_{r,R})-\kappa_N(\tau_1\wedge\tau^N_{r,R})|\geq \alpha\Big)\leq $$ $$\sup_N \sup_{0\leq\tau_1\leq\tau_2\leq (\tau_1+\delta)\wedge T}
\left\{P\left( \int_{\tau_1\wedge\tau^N_{r,R}}^{\tau_2\wedge\tau^N_{r,R}} |H_N^{r,R}\kappa_N(s)|ds\geq \alpha \right) + P\left( \left| \mathcal{M}_{N,\kappa}^{(\tau_1,\tau_2)\wedge\tau^N_{r,R}} \right|  \geq \alpha\right)\right\} \leq \varepsilon $$ which proves \eqref{tight2}.
\end{proof}

\subsection{Averaging principle}
In the this paragraph we prove an elementary averaging principle that is built \textit{ad hoc} for our purpose.

\begin{proposition}\label{averagingproposition} Consider $\phi:\mathbb{R}\times\mathbb{R}\to\mathbb{R}$ a locally Lipschitz continuous function, $2\pi$-periodic in the second variable. Let $\{(x_N(t),y_N(t))_{t\in[0,T]}\}_{N\geq 1}$ be a family of cadlag Markov processes such that: \begin{enumerate}
\item as $N\to\infty$, $(x_N(t))_{t\in[0,T]}$ converges, in sense of weak convergence of stochastic processes, to a process $(\bar{x}(t))_{t\in[0,T]}.$ Assume also that there exists a compact set $K\subset \mathbb{R}$ such that, for any $t\in[0,T]$ and $N\geq 1$, $x_N(t)\in K$ and $\bar{x}(t)\in K$ and that condition \eqref{tight2} holds true for the sequence $(x_N(t))_{N\geq 1}$;
\item for any $\gamma>0$ there exist $h'>0$ and $ \bar{N}\geq 1,$ such that \begin{equation}
\label{etacond} \sup_{0\leq h \leq h'} E\left[ \left| y_N(t+h)-y_N(t) \right| \right] \leq\gamma
\end{equation} for any $N\geq \bar{N}$ and $t\in[0,T]$. 
\end{enumerate}  Then, for any $c>0$ and $\xi>0$, the following averaging principle holds: \begin{equation}
\label{averaging} \int_0^T  \phi\left(x_N(s), cN^\xi s + y_N(s)\right)ds \xrightarrow{weakly} \int_0^T \bar{\phi}\left(\bar{x}(s)\right)ds, \:\:\:\:\: \textrm{as} \:\: N\to\infty
\end{equation} where $\bar{\phi}$ is the averaged function defined by $$\bar{\phi}(x)=\frac{1}{2\pi}\int_0^{2\pi} \phi(x,\theta)d\theta.$$
\end{proposition}
Before proving Proposition \ref{averagingproposition}, recall the Skorohod's Representation Theorem (see \cite{B99}).
\begin{theorem}\label{skorohod}
Suppose that the sequence of probability measures $(P_N)_{N\geq 1}$ weakly converges to $P$ and $P$ has a separable support. Then there exist random elements $(X_N)_{N\geq 1}$ and $X$, defined on a common probability space $(\Omega,\mathcal{F}, \mathcal{P})$, such that $\mathcal{L}(X_N)=P_N$ for any $N\geq 1$, $\mathcal{L}(X)=P$ and $X_N(\omega)\to X(\omega)$ for every $\omega$.
\end{theorem}

\begin{proof}[Proof of Proposition \ref{averagingproposition}]Let $P_N=\mathcal{L}(x_N)$, $\tilde{P}_N=\mathcal{L}((x_N,y_N))$  for any $N\geq 1$ and $P=\mathcal{L}(\bar{x})$ : then, $P_N$ and $P$ are probability measures over the set of cadlag trajectories $\mathcal{D}([0,T],\mathbb{R})$, which, endowed with the Skorohod topology, is a complete and separable metric space (see \cite{EK86}). Therefore, by the Skorohod's Representation Theorem, there exists a probability space on which are defined  $\mathcal{D}([0,T],\mathbb{R})$-valued random variables $x_N$ with distribution $P_N$, for $N\geq 1$, and $\bar{x}$ with distribution $P$ such that $x_N$ converges to $\bar{x}$ a.s.. Notice that on this common probability space (or at least enlarging it) we can also define, for any $N\geq 1$, a random variable $y_N$ such that the joint distribution of $(x_N,y_N)$ is $\tilde{P}_N$. In the following, we will prove that, on this common probability space, $$\int_0^T\phi\left(x_N(s), cN^\xi s + y_N(s)\right)ds \xrightarrow{L^1} \int_0^T \bar{\phi}\left(\bar{x}(s)\right)ds, \:\:\:\:\: \textrm{as} \:\: N\to\infty $$ which implies \eqref{averaging}.\\
First of all, we have that $$ E\left| \int_0^T \phi\left(x_N(s), cN^\xi s + y_N (s)\right)ds - \int_0^T \bar{\phi}\left(\bar{x} (s)\right)ds\right|\leq$$ \begin{equation} 
\leq E\left| \phi\left(x_N (s), cN^\xi s + y_N (s)\right)ds- \int_0^T \bar{\phi}\left(x_N (s)\right)ds\right|+ \label{b1}
\end{equation} \begin{equation} \label{b2}
+E\left|\int_{0}^T\bar{\phi}\left(x_N (s)\right)ds-\int_0^T \bar{\phi}\left(\bar{x} (s)\right)ds\right|.
\end{equation}
Notice that, since $x_N\to\bar{x}$ a.s., than also $$\int_{0}^T\bar{\phi}\left(x_N (s)\right)ds\to\int_0^T\bar{\phi}\left(\bar{x} (s)\right)ds \:\:\:\:\: a.s..$$ Since there exists a compact set such that, for any $s\in[0,T]$, $x_N(s)\in K$ and $\bar{\phi}$ is continuous, the quantities above can be uniformly dominated by a constant hence, by the Dominated Convergence Theorem, we also have convergence in $L^1$ sense so the quantity in \eqref{b2} converges to 0.\\
We have to study \eqref{b1}: for any $n\geq 1$ consider the partition $\mathcal{P}_n$ of $[0,T]$ defined as $$0=t_0<t_1<\dots<t_n\leq t_{n+1}=T$$ where $$n=\left\lfloor \frac{cN^\xi }{2\pi}T\right\rfloor,  \:\:\:\:\:\:\:\: \left| t_i-t_{i-1} \right|=\frac{2\pi}{cN^\xi}\:\:\: \forall\: i=1,\dots,n, \:\:\:\:\:\:\:\:|T-t_n|<\frac{2\pi}{cN^\xi}.$$ Then, $$E\left| \int_0^{T } \phi\left(x_N(s), cN^\xi s + y_N (s)\right)ds - \int_0^{T } \bar{\phi}\left(x_N (s)\right)ds\right|\leq$$ \begin{equation} 
\leq E\left|\sum_{i=0}^{n-1} \int_{t_i }^{t_{i+1} }\phi\left(x_N (s), cN^\xi s + y_N (s)\right)ds- \int_0^{T } \bar{\phi}\left(x_N (s)\right)ds\right|+ \label{a1}
\end{equation} \begin{equation} \label{a2}
+E\int_{t_n }^{T }\left|\phi\left(x_N (s), cN^\xi s + y_N (s)\right)\right|ds
\end{equation}
The term in $\eqref{a2}$ converge to 0: since $x_N (s)\in K$ with $K$ for all $s\in [0,T]$, and $\phi$ is continuous and periodic in its second variable, there exists $C_1>0$ independent of $N$ such that $$\max_{s\in[0,T]} \left| \phi\left(x_N (s), cN^\xi s + y_N (s)\right)\right|\leq C_1.$$ Hence, $$E\left[ \int_{t_n }^{T } \left| \phi\left(x_N (s), cN^\xi s + y_N (s)\right)\right|ds   \right]\leq C_1(T-t_n)\to 0$$ as $n\to \infty$, so we can deal with the term in \eqref{a1} only.
$$E\left|\sum_{i=0}^{n-1} \int_{t_i }^{t_{i+1} }\phi\left(x_N (s), cN^\xi s + y_N (s)\right)ds- \int_0^{T } \bar{\phi}\left(x_N (s)\right)ds\right|\leq $$
\begin{equation}
\label{a3}E \left|
\sum_{i=0}^{n-1} \int_{t_i }^{t_{i+1} } \left[ \phi\left(x_N (s), cN^\xi s + y_N (s)\right) - \phi\left(x_N (t_i), cN^\xi s + y_N (t_i)\right)    \right]ds\right|+
\end{equation} 
\begin{equation}
\label{a4}E \left|  \sum_{i=0}^{n-1} \int_{t_i }^{t_{i+1} } \phi\left(x_N (t_i), cN^\xi s + y_N (t_i)\right)ds - \int_0^{T }\bar{\phi}\left( x_N (s)  \right)ds \right|.
\end{equation}
Notice that, for any $i=0,\dots,n-1$, $$\int_{t_i}^{t_{i+1}}\phi\left(  x_N (t_i),cN^\xi s +y_N (t_i)  \right)ds=\int_{t_i}^{t_{i}+\frac{2\pi}{cN^\xi}}\phi\left(  x_N (t_i),cN^\xi s +y_N (t_i)  \right)ds\overset{\dagger}{=}$$ $$\overset{\dagger}{=} \frac{1}{cN^\xi} \int_0^{2\pi} \phi\left( x_N(t_i),\sigma + cN^\xi t_i + y_N (t_i) \right)d\sigma = \frac{2\pi}{cN^\xi}\bar{\phi}\left(x_N (t_i)\right)=\bar{\phi}\left(x_N (t_i)\right)(t_{i+1}-t_i)$$ where in $\dagger$ we used the change of variable $\sigma= cN^\xi(s-t_i)$. The quantity in \eqref{a4} can therefore be written as $$ E\left| \sum_{i=0}^{n-1} \bar{\phi}(x_N(t_i))(t_{i+1}-t_i)-\int_0^T \bar{\phi}(x_N(s))ds\right| $$ which clearly converges to 0 as $n\to\infty$. \\ So we are only left to prove convergence of the term in \eqref{a3}:
$$ E\left|
\sum_{i=0}^{n-1} \int_{t_i}^{t_{i+1}} \left[ \phi\left(x_N (s), cN^\xi s + y_N (s)\right) - \phi\left(x_N (t_i), cN^\xi s + y_N (t_i)\right)    \right]ds\right| \leq $$
\begin{equation}
\label{a5}= \sum_{i=0}^{n-1} \int_{t_i}^{t_{i+1}} E\left| \phi\left(x_N (s), cN^\xi s + y_N (s)\right) - \phi\left(x_N (t_i), cN^\xi s + y_N (s)\right)    \right|ds+
\end{equation} \begin{equation}
\label{a6} +\sum_{i=0}^{n-1} \int_{t_i}^{t_{i+1}} E\left| \phi\left(x_N (t_i), cN^\xi s + y_N (s)\right) - \phi\left(x_N (t_i), cN^\xi s + y_N (t_i)\right)    \right|ds .
\end{equation}
Notice that the function  if $\phi$ is locally Lipschitz continuous and periodic in the second variable, then it is also locally Lipschitz continuous in the first variable uniformly in the second variable. Since $x_N(s)\in K$ for all $s\in[0,T]$, there exists $L_K>0$ such that $$\sum_{i=0}^{n-1} \int_{t_i}^{t_{i+1}}E \left| \phi\left(x_N (s), cN^\xi s + y_N (s)\right) - \phi\left(x_N (t_i), cN^\xi s + y_N (s)\right)    \right|ds\leq$$ $$\leq L_K \sum_{i=0}^{n-1}\int_{t_i}^{t_{i+1}} E|x_N(s)-x_N(t_i)|ds.$$ Since \eqref{tight2} holds for the sequence $(x_N)_{N\geq 1}$, for any $\varepsilon,\alpha>0$ there exists $\delta >0$ such that $$ \sup_N\sup_{0\leq\tau_1\leq\tau_2\leq(\tau_1+\delta)\wedge T} P\left( |x_N(\tau_2)-x_N(\tau_1)|\geq \alpha \right) \leq \varepsilon.$$
Then, for any $N$ such that ${2\pi\over cN^{\xi}}<\delta$, for any $i=0,\dots,n-1$ and $s\in[t_i,t_{i+1}]$ we have $$E|x_N(s)-x_N(t_i)|\leq $$ $$\leq E\left[ |x_N(s)-x_N(t_i)|\mathbbm{1}_{ \{|x_N(s)-x_N(t_i)|\geq\alpha\} }\right]+E\left[ |x_N(s)-x_N(t_i)|\mathbbm{1}_{\{|x_N(s)-x_N(t_i)|<\alpha\}}\right]\leq$$ $$\leq \max_{x,y\in K} |x-y|P(|x_N(s)-x_N(t_i)|\geq\alpha)+\alpha\leq  \bar{C} \varepsilon+\alpha $$ where $\bar{C}:=\max_{x,y\in K}|x-y|.$ So, for any $\varepsilon,\alpha>0$ and $N$ large enough, $$\sum_{i=0}^{n-1} \int_{t_i}^{t_{i+1}}E \left| \phi\left(x_N (s), cN^\xi s + y_N (s)\right) - \phi\left(x_N (t_i), cN^\xi s + y_N (s)\right)    \right|ds\leq$$ $$\leq L_K \sum_{i=0}^{n-1}\int_{t_i}^{t_{t+1}} \bar{C}\varepsilon+\alpha ds\leq L_K(\bar{C}\varepsilon+\alpha)\left\lfloor \frac{cN^\xi}{2\pi}T \right\rfloor\frac{2\pi}{cN^{\xi}}\leq L_K(\bar{C}\varepsilon+\alpha)T.$$
Therefore, the quantity in \eqref{a5} converges to 0.\\
Finally, we are left with \eqref{a6}: for any $\gamma>0$ and $N$ large enough, $$\sum_{i=0}^{n-1} \int_{t_i}^{t_{i+1}} E\left| \phi\left(x_N (t_i), cN^\xi s + y_N (s)\right) - \phi\left(x_N (t_i), cN^\xi s + y_N (t_i)\right)    \right|ds \leq$$
$$ \leq L_K \sum_{i=0}^{n-1} \int_{t_i}^{t_{i+1}} E\left|y_N (s)-y_N (t_i)\right| ds\leq $$  $$\leq L_K \sum_{i=0}^{n-1} \int_{t_i}^{t_{i+1}} \sup_{0\leq h\leq \frac{2\pi}{cN^\xi}} E\left|y_N (t_i+h)-y_N (t_i)\right| ds\leq$$ $$\leq L_K \sum_{i=0}^{n-1} \int_{t_i}^{t_{i+1}} \gamma ds \leq \left\lfloor \frac{cN^\xi}{2\pi} T \right\rfloor \frac{2\pi}{cN^\xi}L_K\gamma \leq L_KT\gamma   $$
where we used Lipschitz continuity and condition \eqref{etacond}. Since $\gamma$ is arbitrary, the proof is completed.
\end{proof}
\subsection{Analysis of the process $\theta_N(t)$}
Consider again the process $(\kappa_N(t),\theta_N(t))_{t\in[0,T]}$ defined in Lemma \ref{HN} and define the process $(\eta_N(t))_{t\in[0,T]}$ by \begin{equation}
\label{etadef} \eta_N(t):=\theta_N(t)-N^{1\over 2}2\sqrt{\beta-1}t.
\end{equation} The following result proves that, given the stopping time $$\tau^N_{r,R}:=\inf\{t\in[0,T]\:|\: \kappa_N(t)\not\in\:]r,R[ \}, $$ the stopped process $\eta_N(t\wedge\tau^N_{r,R})$ satisfies (a stronger version of) condition \eqref{etacond}.

\begin{proposition} \label{etavera}
Let $\eta_N(t)$ be the Markov process defined by \eqref{etadef}. Then, for $h'>0$, there exist $C>0$ and $\bar{N}\geq 1$ such that \begin{equation}
\label{etavera2} \sup_{0\leq h \leq h'} E\Big[ \left| \eta_N((t+h)\wedge\tau^N_{r,R})- \eta_N(t\wedge\tau^N_{r,R}) \right| \Big]\leq C\sqrt{h'}
\end{equation} for all $N\geq \bar{N}.$
\end{proposition}
\begin{proof}
Notice that $(\kappa_N(t),\eta_N(t))$ is a time-dependent invertible transformation of $(z_N(t),u_N(t))$: so, $(\kappa_N(t),\eta_N(t))$ is itself a (time inhomogeneous) Markov process. We want to find an expression for its generator $J_{N,t}$. Actually, in order to overcome time-dependence, we let time play the role of additional variable: let $(y_N(t))_{t\in[0,T]}$ be the process $$dy_N(t)=N^{1\over 2} 2\sqrt{\beta-1}dt$$ and consider the transformation $$\kappa_N(t)=z^2_N(t)+u^2_N(t),$$
$$\eta_N(t)=\arctan\left(\frac{u_N(t)}{z_N(t)}\right)-y_N(t),$$ $$\xi_N(t)=y_N(t).$$
Recall generator $G_N$ introduced in Lemma \ref{gn}: then, infinitesimal generator $\bar{G}_N$ associated with $(z_N(t),u_N(t),y_N(t))$ is given by $$\bar{G}_N f(z,u,y)=G_Nf(z,u)+N^{1\over 2}2\sqrt{\beta-1}f_y.$$ Let $g$ be the function $$g(z,u,y)=\left( z^2+u^2,\arctan\left({u\over z}\right) - y,y\right)$$ and compute $\bar{G}_N(f\circ g)(z,u,y)$, with $f\in\mathcal{C}^3_b([r,R]\times\mathbb{R}^2)$. Very similar computations to the one performed in the proof of Lemma \ref{HN} yield us \begin{align*}
\bar{G}_N(f\circ g)(z,u,y)&=2\beta^2(f\circ g)_{zz} - \left( \frac{2\beta z^3}{3} + N^{1\over 2} 2\beta \sqrt{\beta-1}u \right)(f\circ g)_z +\\ &+ N^{1\over 2} 2\sqrt{\beta-1}z(f\circ g)_u + N^{1\over 2}2\sqrt{\beta-1}f_y + o(1).
\end{align*}
Observe that: $$(f\circ g)_z=2z f_\kappa - \frac{u}{z^2+u^2} f_\eta,\:\:\:\:\:\:\: (f\circ g)_u=2uf_\kappa + \frac{z}{z^2+u^2}f_\eta,\:\:\:\:\:\:\:(f\circ g)_y=-f_\eta+f_\xi,$$
$$ (f\circ g)_{zz}= 4z^2f_{\kappa\kappa} - \frac{4zu}{z^2+u^2}f_{\kappa\eta} + \frac{u^2}{(z^2+u^2)^2}f_{\eta\eta} + 2f_\kappa + \frac{2uz}{(z^2+u^2)^2}f_\eta,$$
hence \begin{align*}
\bar{G}_N(f\circ g)(z,u,y)&=\:8\beta^2 z^2 f_{\kappa\kappa} - \frac{8\beta^2zu}{z^2+u^2}f_{\kappa\eta} +\frac{2\beta^2u^2}{(z^2+u^2)^2}f_{\eta\eta}+ \\ &+\left( \frac{4\beta^2zu}{(z^2+u^2)^2}+\frac{2\beta z^3u}{3(z^2+u^2)} \right)f_\eta+ \\&+ \left( 4\beta^2 -  \frac{4\beta z^4}{3}\right)f_\kappa + N^{1\over 2}2\sqrt{\beta-1}f_\xi+o(1),
\end{align*}
which finally yields, thanks to Lemma \ref{cambiovariabilegeninf}, the infinitesimal generator $J_N$ of $(\kappa_N(t),\eta_N(t),\xi_N(t))$ in the form: \begin{align*}
J_Nf(\kappa,\eta,\xi)&=8\beta^2\kappa\cos^2(\eta+\xi) f_{\kappa\kappa}  -8\beta^2\cos(\eta+\xi)\sin(\eta+\xi) f_{\kappa\eta}+ \frac{2\beta^2\sin^2(\eta+\xi)}{\kappa} f_{\eta\eta}  +\\&+\left(   \frac{4\beta^2\cos(\eta+\xi)\sin(\eta+\xi)}{\kappa}+\frac{2\beta}{3} \kappa\cos^3(\eta+\xi)\sin(\eta+\xi) +2\sqrt{\beta-1}\sin(\eta+\xi)\right)f_\eta +\\&+ \left(4\beta^2-\frac{4\beta}{3}\kappa^2\cos^4(\eta+\xi)\right)f_\kappa +N^{1\over 2}2\sqrt{\beta-1}f_\xi +o(1).
\end{align*}
So, the infinitesimal generator $J^{r,R}_N$ of the stopped process $(\kappa_N(t\wedge\tau^N_{r,R}),\eta_N(t\wedge\tau^N_{r,R}),\xi_N(t\wedge\tau^N_{r,R}))$ will satisfy for $f\in\mathcal{C}^3_b([r,R]\times\mathbb{R}^2)$:
\begin{align*}
J_N^{r,R}f(\kappa,\eta,\xi)&=\mathbbm{1}_{]r,R[}(\kappa)\left[8\beta^2\kappa\cos^2(\eta+\xi) f_{\kappa\kappa}  -8\beta^2\cos(\eta+\xi)\sin(\eta+\xi) f_{\kappa\eta}+ \frac{2\beta^2\sin^2(\eta+\xi)}{\kappa} f_{\eta\eta}  +\right.\\&+\left(   \frac{4\beta^2\cos(\eta+\xi)\sin(\eta+\xi)}{\kappa}+\frac{2\beta}{3} \kappa\cos^3(\eta+\xi)\sin(\eta+\xi) +2\sqrt{\beta-1}\sin(\eta+\xi)\right)f_\eta +\\&+\left. \left(4\beta^2-\frac{4\beta}{3}\kappa^2\cos^4(\eta+\xi)\right)f_\kappa +N^{1\over 2}2\sqrt{\beta-1}f_\xi +o_{r,R}(1)\right]
\end{align*}
where, as in Lemma \ref{HN}, one can easily check that the remainders $o_{r,R}(1)$ can be uniformly dominated by a term of order $o(1)$ on $[r,R]\times\mathbb{R}^2$.\\
We will use the decomposition \begin{equation}
\label{decomposedeta} \eta_N(t\wedge\tau^N_{r,R})=\eta_N(0)+\int_0^{t\wedge\tau^N_{r,R}} J_N^{r,R}\eta_N(s)ds + \mathcal{M}_{N,\eta}^{t\wedge\tau^N_{r,R}}
\end{equation} where $J_N^{r,R}\eta$ indicates the infinitesimal generator $J_N^{r,R}$ evaluated on the function $f(\kappa,\eta,\xi)=\eta$ and $$\mathcal{M}_{N,\eta}^{t\wedge\tau^N_{r,R}} =\int_0^{t\wedge\tau^N_{r,R}} \sum_{j\in\mathcal{S}} \Delta \eta_N(s) \tilde{\Lambda}_N(j,ds)$$ where $\tilde{\Lambda}(j,ds)$ is the same compensated point process introduced in Proposition \ref{tight} and $\Delta\eta_N(s)$ will be estimated in the following.\\
Let $h>0$ and study the quantity $$\left| \eta_N((t+h)\wedge\tau^N_{r,R})-\eta_N(t\wedge\tau^N_{r,R})\right|\leq \int_{t\wedge\tau^N_{r,R}}^{(t+h)\wedge\tau^N_{r,R}} \left|J_N^{r,R}\eta_N(s)\right|ds+\left| \mathcal{M}_{N,\eta}^{(t+h)\wedge\tau^N_{r,R}}- \mathcal{M}_{N,\eta}^{t\wedge\tau^N_{r,R}}\right|.$$
We have $$J_N^{r,R}\eta=\frac{4\beta^2\cos(\eta+\xi)\sin(\eta+\xi)}{\kappa} +\frac{2\beta}{3}\cos^3(\eta+\xi)\sin(\eta+\xi) + 2\sqrt{\beta-1}\sin(\eta+\xi)+o_{r,R}(1),$$ therefore \begin{equation}\label{C6h}\int_{t\wedge\tau^N_{r,R}}^{(t+h)\wedge\tau^N_{r,R}} |J_N^{r,R}\eta_N(s)|ds\leq \int_{t\wedge\tau^N_{r,R}}^{(t+h)\wedge\tau^N_{r,R}}\frac{4\beta^2}{r}+2\beta R + 2\sqrt{\beta-1} + o_{r,R}(1)ds\leq C(r,R)h.\end{equation}
On the other hand, using the same arguments of the proof of Proposition \ref{tight}, $$E\left[ \left( \mathcal{M}_{N,\eta}^{(t+h)\wedge\tau^N_{r,R}}- \mathcal{M}_{N,\eta}^{t\wedge\tau^N_{r,R}}\right)^2\right]=E\left[ \left( \mathcal{M}_{N,\eta}^{(t+h)\wedge\tau^N_{r,R}}\right)^2- \left(\mathcal{M}_{N,\eta}^{t\wedge\tau^N_{r,R}}\right)^2\right]= $$ 
$$ =E\left[\int_{t\wedge\tau^N_{r,R}}^{(t+h)\wedge\tau^N_{r,R}}\sum_{j\in\mathcal{S}}\left(\Delta \eta_N(s)\right)^2N^{1\over 2}|A_N(j,N^{1\over 2}s)|\Big(1-j\tanh(\lambda_N(N^{1\over 2}s)) \Big) ds\right]. $$
We want to estimate the term $(\Delta\eta_N(t))^2$ for $t\in[0,T\wedge\tau^N_{r,R}]$: recall \eqref{etadef}, so that $$\eta_N(t)-\eta_N(t-)=\theta_N(t)-\theta_N(t-)=\arctan\left(\frac{u_N(t)}{z_N(t)}\right)-\arctan\left(\frac{u_N(t-)}{z_N(t-)}\right).$$
As previously remarked, $u_N(t)=u_N(t-)$ for all $t\in[0,T]$ while, if $\tau\in[0,T\wedge\tau^N_{r,R}]$ is a jump time for $z_N$, $$|z_N(\tau)-z_N(\tau-)|=\frac{2\beta}{N^{3\over 4}}.$$ Since $|{d\over dx}\arctan(x)|\leq 1$ for all $x\in\mathbb{R}$, \begin{equation}\label{saltoeta}|\eta_N(\tau)-\eta_N(\tau-)|\leq \left| \frac{u_N(\tau)}{z_N(\tau)} - \frac{u_N(\tau-)}{z_N(\tau-)}\right|=\left|u_N(\tau)\frac{z_N(\tau)-z_N(\tau-)}{z_N(\tau)z_N(\tau-)}  \right|.\end{equation} Notice also that using the well-known property of $\arctan$ $$\left|\arctan(x)+\arctan\left(\frac{1}{x}\right)\right|=\frac{\pi}{2} \:\:\:\:\: \forall\: x\not=0,$$ we have that $$\left|\arctan\left(\frac{u_N(\tau)}{z_N(\tau)}\right)-\arctan\left(\frac{u_N(\tau-)}{z_N(\tau-)}\right)\right|=\left|\arctan\left(\frac{z_N(\tau)}{u_N(\tau)}\right)-\arctan\left(\frac{z_N(\tau-)}{u_N(\tau-)}\right)\right|$$ and we get a second estimate \begin{equation}
\label{saltoeta2} |\eta_N(\tau)-\eta_N(\tau-)|\leq \left| \frac{z_N(\tau)}{u_N(\tau)} - \frac{z_N(\tau-)}{u_N(\tau-)}\right|=\left|\frac{z_N(\tau)-z_N(\tau-)}{u_N(\tau)}  \right|.\end{equation}
Let now fix $\varepsilon>0$ and $\bar{N}\geq 1$ such that $2\beta\bar{N}^{-{3\over 4}}<\varepsilon$ and $4\varepsilon^2< r;$ so, for $N>\bar{N}$, if $|z_N(\tau)z_N(\tau-)|\geq \varepsilon^2$ it holds, by \eqref{saltoeta}, $$|\eta_N(\tau)-\eta_N(\tau-)|\leq \frac{\sqrt{R}}{\varepsilon^2}\frac{2\beta^2}{ N^{3\over 4}}.$$ Otherwise, consider the case $|z_N(\tau)z_N(\tau-)|<\varepsilon^2$: this implies $|z_N(\tau)|<2\varepsilon$ hence $|u_N(\tau)|\geq \sqrt{r-4\varepsilon^2}$. 
Therefore, by \eqref{saltoeta2}, $$|\eta_N(\tau)-\eta_N(\tau-)|\leq \frac{1}{\sqrt{r-4\varepsilon^2}}\frac{2\beta}{N^{3\over 4}}.$$ In conclusion, having fixed $\varepsilon$ and $\bar{N}$ as above, set $\bar{C}=\max\left\{ {1\over\sqrt{r-4\varepsilon^2}},{\sqrt{R}\over \varepsilon^2} \right\}$ we have \begin{equation}\label{saltoetavero} 
|\Delta\eta_N(t)|\leq \bar{C}\frac{2\beta}{N^{3\over 4}}
\end{equation} for any $t\in[0,T\wedge\tau^N_{r,R}]$ and $N\geq \bar{N}$. By means of \eqref{saltoetavero} $$E\left[ \left( \mathcal{M}_{N,\eta}^{(t+h)\wedge\tau^N_{r,R}}- \mathcal{M}_{N,\eta}^{t\wedge\tau^N_{r,R}}\right)^2\right]=$$ $$ =E\left[\int_{t\wedge\tau^N_{r,R}}^{(t+h)\wedge\tau^N_{r,R}}\sum_{j\in\mathcal{S}}\left(\Delta \eta_N(s)\right)^2N^{1\over 2}|A_N(j,N^{1\over 2}s)|\Big(1-j\tanh(\lambda_N(N^{1\over 2}s)) \Big) ds\right]\leq $$ \begin{equation*}\leq E\left[\int_{t\wedge\tau^N_{r,R}}^{(t+h)\wedge\tau^N_{r,R}}\sum_{j\in\mathcal{S}} \bar{C}^2 \frac{4\beta^2}{N^{3\over 2}} N^{1\over 2} \sup_{j\in\mathcal{S}}| A_N(j,N^{1\over 2}s)| 2 ds \right]\leq 16\beta^2\bar{C}^2 h,\end{equation*}
which, by H\"{o}lder inequality, implies \begin{equation}
\label{Cbar} E\left[ \left| \mathcal{M}_{N,\eta}^{(t+h)\wedge\tau^N_{r,R}}- \mathcal{M}_{N,\eta}^{t\wedge\tau^N_{r,R}}\right|\right]\leq \sqrt{E\left[ \left( \mathcal{M}_{N,\eta}^{(t+h)\wedge\tau^N_{r,R}}- \mathcal{M}_{N,\eta}^{t\wedge\tau^N_r\wedge\tau^N_R}\right)^2\right]}\leq 4\beta\bar{C}\sqrt{h}.
\end{equation}
In conclusion, using \eqref{C6h} and \eqref{Cbar}, for a given $h'>0$ there exists $\bar{N}\geq 1$ and $\bar{C}>0$ such that $$\sup_{0\leq h \leq h'} E\Big[ \left| \eta_N((t+h)\wedge\tau^N_{r,R})- \eta_N(t\wedge\tau^N_{r,R}) \right| \Big]\leq$$ $$\leq\sup_{0\leq h\leq h'} E\int_{t\wedge\tau^N_{r,R}}^{(t+h)\wedge\tau^N_{r,R}} \left|J_N\eta_N(s)\right|ds+\sup_{0\leq h\leq h'}E\left| \mathcal{M}_{N,\eta}^{(t+h)\wedge\tau^N_{r,R}}- \mathcal{M}_{N,\eta}^{t\wedge\tau^N_{r,R}}\right|\leq$$ $$\leq \sup_{0\leq h\leq h'} C(r,R)h+\sup_{0\leq h\leq h'} 4\beta\bar{C}\sqrt{h}\leq C(r,R)h'+ 4\beta\bar{C}\sqrt{h'},$$ for any $N\geq \bar{N}$. Being $h' \ll \sqrt{h'}$ the proof is completed.
\end{proof}

\subsection{Characterization of the limit of the sequence of stopped processes}
By Proposition \ref{tight}, for any choice of $r,R$ such that $0<r<{\beta \over \beta-1}\bar{\lambda}<R$, the sequence $(P_N\circ\varphi^{-1}_{r,R})_{N\geq 1}$ admits a converging subsequence $(P_n\circ\varphi_{r,R}^{-1})_{n\geq 1}$, where $P_N=\mathcal{L}(\kappa_N)$ and $\varphi_{r,R}$ defined by \eqref{phi} (see Remark \ref{stoppedmeasure}). Let $\bar{P}_{r,R}$ be the limit of $(P_n\circ\varphi_{r,R}^{-1})_{n\geq 1}$, which identifies a \textit{cadlag} stochastic process $(\bar{\kappa}_{r,R}(t))_{t\in[0,T]}$. Defining the stopping time $$\tau^n_{r,R}:=\inf\{t\in[0,T]\:|\: \kappa_n(t)\not\in\:] r,R[\}, $$ we have that the sequence of processes $(\kappa_n(t\wedge\tau^n_{r,R}))_{n\geq 1}$ weakly converges to the stochastic process $\bar{\kappa}_{r,R}(t)$. Let now $L$ denote the infinitesimal generator of the solution of the stochastic differential equation $$d\kappa(t)=\left(4\beta^2-\frac{\beta}{2}\kappa^2(t)\right)dt+2\beta\sqrt{2\kappa(t)}dB(t),$$ namely, for any $f\in dom(L)$, \begin{equation}\label{geninfeqlimite}Lf(\kappa)=4\beta^2\kappa f''(\kappa)+\left( 4\beta^2-\frac{\beta}{2}\kappa^2 \right)f'(\kappa).\end{equation} The process $\bar{\kappa}_{r,R}(t)$ satisfies the following property.

\begin{proposition}\label{stopmart}
For any $f\in\mathcal{C}_c^{\infty}([r,R])$ the stochastic process $$\bar{\mathcal{N}}_{r,R}^f(t):=f(\bar{\kappa}_{r,R}(t))-f(\bar{\kappa}_{r,R}(0))-\int_0^t Lf(\bar{\kappa}_{r,R}(s))ds$$ is a martingale.
\end{proposition}
\begin{proof}
Let $f\in\mathcal{C}_c^{\infty}([r,R])$ and fix the following notations: $$A_f(\kappa,\theta):=8\beta^2\kappa\cos^2\theta f''(\kappa) + \left(4\beta^2-\frac{4\beta}{3}\kappa^2\cos^4\theta\right)f'(\kappa),$$ $$\bar{A}_f(\kappa):=\frac{1}{2\pi} \int_0^{2\pi} A_f(\kappa,\theta) d\theta.$$
Notice that it holds that $$H_n^{r,R}f(\kappa,\theta)=A_f(\kappa,\theta)+o_{r,R}(1)$$ where $H_n^{r,R}$ is the infinitesimal generator of the process $\kappa_n(t\wedge\tau^n_{r,R})$ and $\bar{A}_f(\kappa)=Lf(\kappa)$.\\ Define the following processes, for $t\in[0,T]$: $$\mathcal{M}^{f}_n(t)=f(\kappa_n(t\wedge\tau^n_{r,R}))-f(\kappa_n(0))-\int_0^{t\wedge\tau^n_{r,R}} H_n^{r,R}f(\kappa_n(s))ds,$$ $$\mathcal{N}^f_n(t)=f(\kappa_n(t\wedge\tau^n_{r,R}))-f(\kappa_n(0))-\int_0^{t\wedge\tau^n_{r,R}} A_f(\kappa_n(s),\theta_n(s))ds;$$
moreover it holds that
$$\bar{\mathcal{N}}^f_{r,R}(t)=f(\bar{\kappa}_{r,R}(t))-f(\bar{\kappa}_{r,R}(0))-\int_0^{t}\bar{A}_f(\bar{\kappa}_{r,R}(s))ds.$$
We want to show that $\bar{\mathcal{N}}^f_{r,R}(t)$ is a martingale. 
Fix $m \geq 1$, $g_1,\dots,g_m$ continuous and bounded functions and $0\leq t_1 \leq \cdots \leq t_m \leq s \leq t \leq T$. 
Since $\mathcal{M}^f_n$ is a martingale, $$E\Big[\Big( \mathcal{M}^f_n(t)-\mathcal{M}^f_n(s) \Big)g_1(\kappa_n(t_1\wedge\tau^n_{r,R}))\cdots g_m(\kappa_n(t_m\wedge\tau^n_{r,R})) \Big]=0,$$ which implies $$E\Big[\Big( \mathcal{N}^f_n(t)-\mathcal{N}^f_n(s) \Big)g_1(\kappa_n(t_1\wedge\tau^n_{r,R}))\cdots g_m(\kappa_n(t_m\wedge\tau^n_{r,R})) \Big]=o_{r,R}(1).$$
The last equivalence can be written as \begin{equation} 
E\Big[\Big( f(\kappa_n(t\wedge\tau^n_{r,R}))-f(\kappa_n(s\wedge\tau^n_{r,R})) \Big)g_1(\kappa_n(t_1\wedge\tau^n_{r,R}))\cdots g_m(\kappa_n(t_m\wedge\tau^n_{r,R})) \Big]+ \label{pinco1}
\end{equation}
\begin{align*}
 +E\left[-\left(\int_{s\wedge\tau^n_{r,R}}^{t\wedge\tau^n_{r,R}} A_f(\kappa_n(\sigma),\theta_n(\sigma))d\sigma\right) g_1(\kappa_n(t_1\wedge\tau^n_{r,R}))\cdots g_m(\kappa_n(t_m\wedge\tau^n_{r,R})) \right]=
\end{align*}
$$=o_{r,R}(1).$$
Notice that, since $(\kappa_n(t\wedge\tau^n_{r,R}))_{n\geq 1}$ weakly converges to $\bar{\kappa}_{r,R}(t)$, the term in \eqref{pinco1} converges to $$E\Big[\Big( f(\bar{\kappa}_{r,R}(t))-f(\bar{\kappa}_{r,R}(s)) \Big)g_1(\bar{\kappa}_{r,R}(t_1))\cdots g_m(\bar{\kappa}_{r,R}(t_m)) \Big]$$ and $o_{r,R}(1)$ converges to 0,
 in order to show that $\bar{\mathcal{N}}^f_{r,R}(t)$ is a martingale we only have to prove that the term \begin{align*}
 E\left[\left(\int_{s\wedge\tau^n_{r,R}}^{t\wedge\tau^n_{r,R}} A_f(\kappa_n(\sigma),\theta_n(\sigma))d\sigma\right) g_1(\kappa_n(t_1\wedge\tau^n_{r,R}))\cdots g_m(\kappa_n(t_m\wedge\tau^n_{r,R})) \right]
\end{align*} converges to $$E\left[\int_{s}^{t}\bar{A}_f(\bar{\kappa}_{r,R}(\sigma))d\sigma g_1(\bar{\kappa}_{r,R}(t_1))\cdots g_m(\bar{\kappa}_{r,R}(t_m)) \right].$$
Notice that, since $f\in\mathcal{C}_c^\infty([r,R])$, it holds that $$\int_{s\wedge\tau^n_{r,R}}^{t\wedge\tau^n_{r,R}} A_f(\kappa_n(\sigma),\theta_n(\sigma))d\sigma=\int_{s}^{t} A_f(\kappa_n(\sigma\wedge\tau^n_{r,R}),\theta_n(\sigma\wedge\tau^n_{r,R}))d\sigma,$$ so we have all the elements to apply the averaging principle of Proposition \ref{averagingproposition}: we can take $x_N(t)=\kappa_N(t\wedge\tau^N_{r,R})$, $\bar{x}(t)=\bar{\kappa}_{r,R}(t)$, $y_N(t)=\eta_N(t\wedge\tau^N_{r,R})$ with $\eta_N(t)$ defined by \eqref{etadef}, $c=2\sqrt{\beta-1}$ and $\xi={1\over 2}$. Thanks to Proposition \ref{etavera} we can apply the averaging principle, from which $$E\left[ \int_{s}^{t} A_f(\kappa_n(\sigma\wedge\tau^n_{r,R}),\theta_n(\sigma\wedge\tau^n_{r,R}))d\sigma  \right]\to E\left[  \int_{s}^{t} \bar{A}_f(\bar{\kappa}_{r,R}(\sigma))d\sigma \right] $$
as $n\to\infty$. Finally, since  
 $$\int_{s}^{t} A_f(\kappa_n(\sigma\wedge\tau^n_{r,R}),\theta_n(\sigma\wedge\tau^n_{r,R}))d\sigma$$
can be bounded by a constant uniformly in $n$, $$E\left[ \left(\int_{s}^{t} A_f(\kappa_n(\sigma\wedge\tau^n_{r,R}),\theta_n(\sigma\wedge\tau^n_{r,R}))d\sigma\right)g_1(\kappa_n(t_1\wedge\tau^n_{r,R}))\cdots g_m(\kappa_n(t_m\wedge\tau^n_{r,R}))  \right]$$ $$\downarrow$$  $$E\left[ \left( \int_{s}^{t} \bar{A}_f(\bar{\kappa}_{r,R}(\sigma))d\sigma \right)g_1(\bar{\kappa}_{r,R}(t_1))\cdots g_m(\bar{\kappa}_{r,R}(t_m))\right] $$ as $n\to\infty$ and the proof is concluded.
\end{proof}

Let us recall the definition of \textit{stopped martingale problem} and a very useful theorem related to it.
\begin{definition}
Let $E$ be a metric space, $U$ an open subset of $E$ and $(X(t))_{t\in[0,T]}$ be a stochastic process with sample paths in $\mathcal{D}([0,T],E)$. Define the stopping time $$\tau=\inf\{t\in[0,T]\:|\: X(t)\not\in U \:\:\:\: or\:\:\:\: X(t-)\not\in U\}.$$ Let $L$ be a linear operator $L:dom(L)\subset B(E)\to B(E)$ with $B(E)$ the Banach space of bounded measurable function $E\to\mathbb{R}$. Then, $X(\cdot\wedge\tau)$ is a solution of the stopped martingale problem for $(L,U)$ if $$f(X(t\wedge\tau))-f(X(0))-\int_0^{t\wedge\tau} Lf(X(s))ds$$ is a martingale for all $f\in dom(L)$.
\end{definition}

\begin{theorem}[\cite{EK86}, Thm 6.1 p.216]\label{localization}
Let $(E,d)$ be a complete and separable metric space and let $L$ be a linear operator $L:\mathcal{C}_b(E)\to B(E)$. If the $\mathcal{D}([0,T],E)$ martingale problem for $L$ is well-posed, then for any open set $U\subset E$ there exists a unique solution of the stopped martingale problem for $(L,U)$.
\end{theorem}

\begin{proposition}\label{killedprocess}
Fix $\varepsilon>0$: let $\bar{\kappa}_{r,R}(t)$ be the weak limit of the sequence of stopped processes $(\kappa_n(t\wedge\tau^n_{r,R}))_{n\geq 1}$ and define the stopping time $$\bar{\tau}^\varepsilon_{r,R}=\inf\{t\in [0,T]\:|\: \bar{\kappa}_{r,R}(t)\not\in\:\: ]r+\varepsilon,R-\varepsilon[\:\:\},$$with $r,R$ such that $0<r<{\beta\over \beta-1}\bar{\lambda}^2<R$.  Let $\kappa(t)$ be the unique solution of the stochastic differential equation $$d\kappa(t)=\left(4\beta^2-\frac{\beta}{2}\kappa^2(t)\right)dt+2\beta\sqrt{2\kappa(t)}dB(t)$$ with $\kappa(0)={\beta\over \beta-1}\bar{\lambda}^2$ and define $$\tau^\varepsilon_{r,R}=\inf\{t\in [0,T]\:|\: \kappa(t)\not\in\:\: ]r+\varepsilon,R-\varepsilon[\:\:\}.$$ Then, the processes $\bar{\kappa}_{r,R}(t\wedge\bar{\tau}^\varepsilon_{r,R})$ and $\kappa(t\wedge\tau^{\varepsilon}_{r,R})$ have the same distribution.
\end{proposition}

\begin{proof}
First of all, we have that $\kappa_N(0)$ converges in distribution to $\kappa(0).$
Notice that, given $g\in\mathcal{C}_0^{\infty}([0,+\infty[)$, for any $\varepsilon>0$ there exists a function $f\in\mathcal{C}_c^\infty([r,R])$ such that $f(x)=g(x)$ for any $x\in[r+\varepsilon,R-\varepsilon]$. Then, by Proposition \ref{stopmart}, for any $g\in\mathcal{C}_0^{\infty}([0,+\infty[)$, the process $$g(\bar{\kappa}_{r,R}(t\wedge\bar{\tau}^\varepsilon_{r,R}))-g(\bar{\kappa}_{r,R}(0))-\int_0^{t\wedge\bar{\tau}^\varepsilon_{r,R}} Lg(\bar{\kappa}_{r,R}(s))ds$$ is a martingale. \\
On the other hand, the process $\kappa(t\wedge\tau^\varepsilon_{r,R})$  obviously solves the stopped martingale problem for $(L,]r+\varepsilon,R-\varepsilon[)$ where $L$ is given by $\eqref{geninfeqlimite}$. By Lemma \ref{existenceanduni}, the martingale problem for $L$ is well-posed hence for Theorem \ref{localization} the stopped martingale problem for $(L,]r+\varepsilon,R-\varepsilon[)$ has a unique solution. But, since $\mathcal{C}_0^{\infty}([0,+\infty[)$ is measure-determining and for all  $g\in\mathcal{C}_0^{\infty}([0,+\infty[)$ the process $$g(\bar{\kappa}_{r,R}(t\wedge\bar{\tau}^\varepsilon_{r,R}))-g(\bar{\kappa}_{r,R}(0))-\int_0^{t\wedge\bar{\tau}^\varepsilon_{r,R}} Lg(\bar{\kappa}_{r,R}(s))ds$$ is a martingale, the processes $\kappa(t\wedge\tau^\varepsilon_{r,R})$ and $\bar{\kappa}_{r,R}(t\wedge\bar{\tau}^\varepsilon_{r,R})$ must have the same distribution.
\end{proof}
\subsection{From localization to the proof of Theorem \ref{mainthm}}
In this paragraph we exploit the localization argument to conclude the proof of Theorem \ref{mainthm}.

\begin{proposition}\label{proposizionefinale}
Given Proposition \ref{killedprocess} and Lemma \ref{existenceanduni}, the sequence of processes $(\kappa_N(t))_{N\geq 1}$ weakly converges to the unique solution of \eqref{eqlimite}.
\end{proposition}
\begin{proof}
Let us fix some notation: for any $m\geq 1$ define $$U_m=]r_m+\varepsilon_m,R_m-\varepsilon_m[$$ such that $0<r_m<r_m+\varepsilon_m<{\beta\over \beta-1}\bar{\lambda}^2<R_m-\varepsilon_m<R_m$ for all $m\geq 1$, $U_1\subset U_2 \subset \cdots $ and $$\lim_{m\to\infty}r_m+\varepsilon_m=0, \:\:\:\:\:\: \lim_{m\to\infty} R_m-\varepsilon_m =+\infty.$$
For any $m$, let $\bar{\kappa}_m(t):=\bar{\kappa}_{r_m,R_m}(t)$ be the weak limit of $\kappa_N(t\wedge\tau^N_{r_m,R_m})$ and let $\kappa(t)$ be the solution of the limiting equation \eqref{eqlimite}. Moreover, we define the stopping times 
\[
\tau_{U_m}^N=\inf\{t\in [0,T]\:|\: \kappa_N(t) \not\in U_m \},
\]
\[
\bar{\tau}_m=\inf\{t\in [0,T]\:|\: \bar{\kappa}_m(t) \not\in U_m \},
\]
\[
\tau_m=\inf\{t\in [0,T]\:|\: \kappa(t) \not\in U_m \}.
\]
By Proposition \ref{killedprocess}, the process $\bar{\kappa}_m(t)$ is continuous a.s., hence the weak convergence of $(\kappa_N(t\wedge\tau^N_{r_m,R_m}))_{N\geq 1}$  to $\bar{\kappa}_m(t)$ also holds endowing the space $\mathcal{D}([0,T],\mathbb{R})$ with the uniform topology (see for example Lemma 1.6.4 in \cite{S04}). For any $m\geq 1$ consider the set 
\[
A_m=\{x\in\mathcal{D}([0,T],\mathbb{R})\:|\: x(t)\in U_m \:\:\: \forall\:t\in[0,T]\};
\]
since $A_m^c$ is a closed set in the uniform topology of $\mathcal{D}([0,T],\mathbb{R})$ for any $m$, it holds
\begin{equation}\label{limsupclosedset}
\limsup_N P(\kappa_N(\cdot\wedge\tau^N_{r_m,R_m})\in A^c_m) \leq P(\bar{\kappa}_m(\cdot)\in A_m^c), \:\:\:\:\: \forall\:m\geq 1
\end{equation}
by the Portmanteau Theorem. By Proposition \ref{killedprocess} $\bar{\kappa}_m(t\wedge\bar{\tau}_m)$ and $\kappa(t\wedge\tau_m)$ have the same distribution, so observe that 
\[
P(\bar{\kappa}_m(\cdot)\in A_m^c)=P(\bar{\tau}_m\leq T)=P(\tau_m\leq T).
\]
Moreover, since $U_m\subset\: ]r_m,R_m[$,
\[ 
P(\kappa_N(\cdot\wedge\tau^N_{r_m,R_m})\in A^c_m)= P(\tau^N_{U_m}\leq T) \geq P(\tau^N_{r_m,R_m}\leq T),
\]
hence, by \eqref{limsupclosedset},
\begin{equation}
\label{limsupstopping}
\limsup_N P(\tau^N_{r_m,R_m}\leq T)\leq P({\tau}_m\leq T), \:\:\:\:\: \forall\:m\geq 1.
\end{equation}
Let $f:\mathcal{D}([0,T],\mathbb{R})\to \mathbb{R}$ be a continuous and bounded function: we want to show that $$|E[f(\kappa_N(\cdot))]-E[f(\kappa(\cdot))]|\to 0$$ as $N\to\infty$.
We can see that, for any $m,N\geq 1$, 
\begin{align*}
|E[f(\kappa_N(\cdot))]-E[f(\kappa(\cdot))|&\leq| E[f(\kappa_N(\cdot))]-E[f(\kappa_N(\cdot\wedge \tau^N_{r_m,R_m})]| + \\
&+| E[f(\kappa_N(\cdot\wedge \tau^N_{r_m,R_m}))]-E[f(\bar{\kappa}_m(\cdot))]| + \\
&+|E[f(\bar{\kappa}_m(\cdot))]-E[f(\bar{\kappa}_m(\cdot\wedge\bar{\tau}_m))]|+\\
&+|E[f(\bar{\kappa}_m(\cdot\wedge\bar{\tau}_m))]-E[f(\kappa(\cdot\wedge\tau_m))]|+\\
&+|E[f(\kappa(\cdot\wedge\tau_m))-E[f(\kappa(\cdot))]|.
\end{align*}
By Proposition \ref{killedprocess}, the processes $\bar{\kappa}_m(t\wedge\bar{\tau}_m)$ and $\kappa(t\wedge\tau_m)$ have the same distribution, so 
\[
|E[f(\bar{\kappa}_m(\cdot\wedge\bar{\tau}_m))]-E[f(\kappa(\cdot\wedge\tau_m))]|=0
\]
and 
\begin{align*}
|E[f(\bar{\kappa}_m(\cdot))]-E[f(\bar{\kappa}_m(\cdot\wedge\bar{\tau}_m))]|+|E[f(\kappa(\cdot\wedge\tau_m))-E[f(\kappa(\cdot))]|&\leq 2||f||_{\infty}P(\tau_m\leq T).
\end{align*}
Notice also that 
\[
| E[f(\kappa_N(\cdot))]-E[f(\kappa_N(\cdot\wedge \tau^N_{r_m,R_m})]|\leq ||f||_{\infty}P(\tau_{r_m,R_m}^N\leq T),
\]
so we get, for any $m,N\geq 1$,
\begin{align*}
|E[f(\kappa_N(\cdot))]-E[f(\kappa(\cdot))|&\leq||f||_{\infty}P(\tau_{r_m,R_m}^N\leq T) + \\
&+| E[f(\kappa_N(\cdot\wedge \tau^N_{r_m,R_m}))]-E[f(\bar{\kappa}_m(\cdot))]| + \\
&+2||f||_{\infty}P(\tau_m\leq T).
\end{align*}
Fix $\varepsilon>0$: by Lemma \ref{existenceanduni} we have that there exists $m$ large enough such that 
$ P(\tau_m\leq T)<\varepsilon$. After having chosen such $m$, by \eqref{limsupstopping} and the convergence of 
 $\kappa_N(t\wedge \tau^N_{r_m,R_m})$ to $\bar{\kappa}_m(t)$, we can choose $N$ large enough such that
 $P(\tau_{r_m,R_m}^N\leq T)\leq 2\varepsilon$ and 
 \[
 | E[f(\kappa_N(\cdot\wedge \tau^N_{r_m,R_m}))]-E[f(\bar{\kappa}_m(\cdot))]|<\varepsilon.
 \]
So, for any $\varepsilon>0$, there exist $m$ and $N$ large enough such that
\[ 
|E[f(\kappa_N(\cdot))]-E[f(\kappa(\cdot))|\leq \varepsilon+4\varepsilon||f||_{\infty}
\]
which concludes the proof.
\end{proof}

\section{Sketch of proof of Theorem \ref{maintheorem2pop}}
As pointed out in the proof of Theorem \ref{mainthm}, we are looking for  a change of variables 
\begin{equation*}
 \left(\begin{array}{c}
w_N(t)\\v_N(t)
\end{array}\right)=C\left(\begin{array}{c}
x_N(t)\\y_N(t)
\end{array}\right)
\end{equation*}
where $C$ has to be such that $$CAC^{-1}=\left( \begin{array}{cc}
0 & -2\sqrt{|\Gamma|}\\2\sqrt{|\Gamma|} & 0
\end{array} \right).$$
It's easy to check that we can take 
$$C=\left( \begin{array}{cc}
0 & {1\over (1-\gamma)J_{21}}\\ -{1\over \sqrt{|\Gamma|}} & {(\gamma J_{11}-1)\over (1-\gamma)J_{21}\sqrt{|\Gamma|} }
\end{array} \right),$$
which provides 
\begin{equation}\label{WV2pop}
w_N(t)= {y_N(t)\over (1-\gamma)J_{21}}, \:\:\:\:\:\:\:\:\:\: v_N(t)= {1\over \sqrt{|\Gamma|}}\left(-x_N(t) + {(\gamma J_{11}-1)\over (1-\gamma)J_{21}} y_N(t)\right).
\end{equation}
Using the same tools of Subsection \ref{espansionegeneratori},  the infinitesimal generator $G_N$ of the process $(w_N(t),v_N(t))$ satisfies, for $f\in \mathcal{C}^3$, 
\begin{align*}
G_Nf(w,v)&= {2\over (1-\gamma) J_{21}^2} f_{ww} + {4(\gamma J_{11}-1)\over (1-\gamma)J_{21}^2 \sqrt{|\Gamma|}} f_{wv} + {2(\gamma(1-\gamma) J_{21}^2+(\gamma J_{11}-1)^2)\over (1-\gamma)J_{21}^2|\Gamma|}  f_{vv}+\\
&+{1\over (1-\gamma)J_{21}}\left( {(1-\gamma)\over 3}\mathcal{R}_2^3(x,y)-y\mathcal{R}^2_2(x,y)  \right)f_w+\\
&+\left({\gamma J_{11}-1\over (1-\gamma)J_{21}\sqrt{|\Gamma|} }\left( {(1-\gamma)\over 3}\mathcal{R}_2^3(x,y)-y\mathcal{R}^2_2(x,y)  \right)-{1\over \sqrt{|\Gamma|}}\left({ \gamma\over 3}\mathcal{R}_1^3(x,y)-x\mathcal{R}_1^2(x,y) \right)\right)f_v+\\
&+N^{1\over 2} \left(  -2\sqrt{|\Gamma|}vf_w+2\sqrt{|\Gamma|}wf_v  \right)+o(1),
\end{align*}
where $x=(\gamma J_{11}-1)w-\sqrt{|\Gamma|}v$, $y=(1-\gamma)J_{21}w$.
This means that for $0<r<R$ and $$\tau^N_{r,R}:=\inf\{t\in[0,T]|\kappa_N(t)\not\in ]r,R[\},$$ the stopped process
 $(\kappa_N(t\wedge\tau_{r,R}^N),\theta_N(t\wedge\tau_{r,R}^N))$ has an infinitesimal generator $H^{r,R}_N$ which, for $f\in \mathcal{C}^3$,   satisfies 
 $$
 Hf_N^{r,R} f(\kappa,\theta) = \mathcal{H}^{r,R}_f(\kappa,\theta) + N^{1\over 2} 2\sqrt{|\Gamma|} f_{\theta}(\kappa,\theta)+o_{r,R}(1)
 $$
 where $\mathcal{H}^{r,R}_f(\kappa,\theta)$ is composed by terms of order 1.
If we apply $\mathcal{H}^{r,R}$ to a function of the type $f(\kappa,\theta)=f(\kappa)$ we get $\mathcal{H}^{r,R}_f(\kappa)=\mathbbm{1}_{]r,R[}(\kappa) \mathcal{A}_f(\kappa,\theta)$ with
\begin{align*}
\mathcal{A}_f(\kappa,\theta)&=\left(  {8\over (1-\gamma)J^2_{21}}w^2+ {16(\gamma J_{11} -1)\over (1-\gamma)J^2_{21}\sqrt{|\Gamma|}}wv +{8 (\gamma(1-\gamma)J^2_{21} + (\gamma J_{11}-1)^2)\over (1-\gamma)J^2_{21}|\Gamma|}v^2\right)f''(\kappa)+\\
&+4 {|\Gamma|+\gamma(1-\gamma)J^2_{21} + (\gamma J_{11}-1)^2\over (1-\gamma)J^2_{21}|\Gamma|}f'(\kappa)+\\
&+\left({2w\over (1-\gamma)J_{21}} + {2(\gamma J_{11} -1)v\over (1-\gamma)J_{21} \sqrt{|\Gamma|}}\right)\left( {(1-\gamma)\over 3}\mathcal{R}_2^3(x,y)-y\mathcal{R}^2_2(x,y)  \right)f'(\kappa)   +\\
&-{2v\over \sqrt{|\Gamma|}}\left({ \gamma\over 3}\mathcal{R}_1^3(x,y)-x\mathcal{R}_1^2(x,y) \right)f'(\kappa),
\end{align*}
where $x=(\gamma J_{11}-1)w-\sqrt{|\Gamma|}v$, $y=(1-\gamma)J_{21}w$, $w=\sqrt{\kappa}\cos\theta$ and $v=\sqrt{\kappa}\sin\theta$.
Now we have to calculate $\bar{\mathcal{A}}_f(\kappa)={1\over 2\pi} \int_0^{2\pi}\mathcal{A}_f(\kappa,\theta)d\theta$: notice that the averaging will make any term with odd power in $w$ or $v$ disappear since 
$$\int_0^{2\pi} \cos^n\theta\sin^m\theta d\theta=0$$ if at least one between $n$ and $m$ is odd. Then, one gets
\begin{align*}
\bar{\mathcal{A}}_f(\kappa)&= 4Z_{1}(\gamma, J_{11}, J_{12}, J_{21}) \left(\kappa f''(\kappa) +f'(\kappa)\right)+{1\over 4}Z_2(\gamma,J_{11},J_{12},J_{21})\kappa^2f'(\kappa)
\end{align*}
where 
\begin{align*}
Z_{1}(\gamma, J_{11}, J_{12}, J_{21})= {|\Gamma|+\gamma(1-\gamma)J^2_{21} + (\gamma J_{11}-1)^2\over (1-\gamma)J^2_{21}|\Gamma|}
\end{align*}
and 
\begin{align*}
Z_2(\gamma,J_{11},J_{12},J_{21})&=-2 J_{11}^2 |\Gamma| - 2 J_{21}^2 + (\gamma J_{11}-1)|\Gamma|(J_{11}^2-J_{21}^2) + (\gamma J_{11} -1) J_{21}^2 \\
&+ (\gamma J_{11}-1) (J_{11}( \gamma J_{11}-1) + (1-\gamma) J_{12} J_{21} )^2 +\\&- J_{11}(\gamma J_{11}-1)(J_{11}( \gamma J_{11}-1) + (1-\gamma) J_{12} J_{21} ) .
\end{align*}
Then, as long as the parameters are chosen in a way that $Z_2(\gamma,J_{11},J_{12},J_{21})<0$, $\bar{\mathcal{A}}_f(\kappa)$ corresponds to the infinitesimal generator of the process $(\kappa(t))_{t\in[0,T]}$, unique solution of \eqref{limitfluttuazioni2pop}. The arguments to be used are the same of the proof of Theorem \ref{mainthm}:
we fix $r,R$ such that $0< r<4\gamma^2|\Gamma|^{-1}<R$ and we characterize the weak limit for the stopped process $(\kappa_N(t\wedge\tau_{r,R}^{N}))_{t\in[0,T]}$ using the averaging principle; then, as in Proposition \ref{proposizionefinale}, one proves that it also implies that $(\kappa_N(t))_{t\in [0,T]}$ converges, in sense of weak convergence of stochastic processes, to $(\kappa(t))_{t\in[0,T]}$.

\bibliographystyle{abbrv}

\begin{thebibliography}{1}

\bibitem{BTW}
Bak, P., Tang, C.,  Wiesenfeld, K. (1988). Self-organized criticality. Physical review A, 38(1), 364.

\bibitem{B99}
P.~Billingsley.
\newblock {\em Convergence of probability measures}.
\newblock Wiley Series in Probability and Statistics: Probability and
  Statistics, John Wiley \& Sons Inc., second edition, 1999.
  
  \bibitem{Bon1}
Bonilla, L. L., Neu, J. C.,  Spigler, R. (1992). Nonlinear stability of incoherence and collective synchronization in a population of coupled oscillators. Journal of statistical physics, 67(1-2), 313-330.

\bibitem{Bon2}
Bonilla, L. L., Vicente, C. P.,  Spigler, R. (1998). Time-periodic phases in populations of nonlinearly coupled oscillators with bimodal frequency distributions. Physica D: Nonlinear Phenomena, 113(1), 79-97.

\bibitem{CerfGorny2}
Cerf, R.,  Gorny, M. (2016). A Curie?Weiss model of self-organized criticality. The Annals of Probability, 44(1), 444-478.

\bibitem{CDP12}
F.~Collet and P.~Dai~Pra.
\newblock {The role of disorder in the dynamics of critical fluctuations of
  mean field models}.
\newblock {\em Electronic Journal of Probability}, 26:1--40, 2012.

\bibitem{CFT16}
F.~Collet, M.~Formentin, and D.~Tovazzi.
\newblock Rhythmic behavior in a two-population mean field {I}sing model.
\newblock {\em Physical Review E}, 94(4):042139, 2016.

\bibitem{CE88}
F.~Comets and T.~Eisele.
\newblock Asymptotic dynamics, noncritical and critical fluctuations for a
  geometric long-range interacting model.
\newblock {\em Communications in Mathematical Physics}, 118:531--567, 1988.

\bibitem{CGM}
Contucci, P., Gallo, I.,  Menconi, G. (2008). Phase transitions in social sciences: two-population mean field theory. International Journal of Modern Physics B, 22(14), 2199-2212.


\bibitem{DPFR13}
P.~Dai~Pra, M.~Fischer, and D.~Regoli.
\newblock {A Curie-Weiss model with dissipation}.
\newblock {\em Journal of Statistical Physics}, 152(1):37--53, 2013.

\bibitem{D83}
D.~A. Dawson.
\newblock Critical dynamics and fluctuations for a mean-field model of
  cooperative behavior.
\newblock {\em Journal of Statistical Physics}, 31(1):29--85, 1983.

\bibitem{EK86}
S.~N. Ethier and T.~G. Kurtz.
\newblock {\em Markov processes: characterization and convergence}.
\newblock Wiley Series in Probability and Statistics: Probability and
  Statistics, John Wiley \& Sons Inc., 1986.

\bibitem{GC}
Gallo, I.,  Contucci, P. (2007). Bipartite mean field spin systems. Existence and solution. Math. Phys. Electron. J., 14(1):1?22

\bibitem{CerfGorny}
Gorny, M. (2015). A Dynamical Curie-Weiss Model of SOC: The Gaussian Case. arXiv preprint arXiv:1507.00924.



\bibitem{IW81}
N.~Ikeda and S.~Watanabe.
\newblock {\em Stochastic differential equations and diffusion processes}.
\newblock North-Holland Publishing Co., Amsterdam, 1981.

\bibitem{S04}
D.~Silvestrov
\newblock {\em Limit Theorems for Randomly Stopped Stochastic Processes}.
\newblock Springer-Verlag London, 2004.
\end{thebibliography}

\end{document}